\documentclass[10pt]{article}
  \usepackage{amsmath}
  \usepackage{amssymb}
  \usepackage{latexsym}
\usepackage{graphicx}
  \usepackage{color}

  \newenvironment{proof}{\vspace{1ex}\noindent{\bf Proof:}}{\hspace*{\fill}$\blacksquare$\vspace{1ex}}
  
  \def\noproof{\hspace*{\fill}$\blacksquare$}

  \newtheorem{theorem}{Theorem}[section]
  \newtheorem{lemma} [theorem] {Lemma}%[section]
  \newtheorem{corollary} [theorem] {Corollary}%[section]
\newcommand{\Acal}[0]{\ensuremath{{\mathcal A}}}

\newcommand{\Dcal}[0]{\ensuremath{{\mathcal D}}}

\newcommand{\Ical}[0]{\ensuremath{{\mathcal I}}}

\newcommand{\Lcal}[0]{\ensuremath{{\mathcal L}}}

\newcommand{\Pcal}[0]{\ensuremath{{\mathcal P}}}
\newcommand{\Qcal}[0]{\ensuremath{{\mathcal Q}}}

\newcommand{\Scal}[0]{\ensuremath{{\mathcal S}}}

\newcommand{\eR}[0]{\ensuremath{ \mathbb R}}
\newcommand{\RP}[0]{\ensuremath{{\eR}P}}
\newcommand{\Qu}[0]{\ensuremath{ \mathbb Q}}
\newcommand{\eN}[0]{\ensuremath{ \mathbb N}}
\newcommand{\Zed}[0]{\ensuremath{ \mathbb Z}}

\newcommand{\ptil}[0]{\tilde{p}}
\newcommand{\qtil}[0]{\tilde{q}}
\newcommand{\rtil}[0]{\tilde{r}}
\newcommand{\xtil}[0]{\tilde{x}}
\newcommand{\ytil}[0]{\tilde{y}}
\newcommand{\Atil}[0]{\tilde{A}}
\newcommand{\Btil}[0]{\tilde{B}}
\newcommand{\Ctil}[0]{\tilde{C}}
\newcommand{\Dtil}[0]{\tilde{D}}
\newcommand{\Etil}[0]{\tilde{E}}

\newcommand{\Otil}[0]{\tilde{O}}

\newcommand{\Qtil}[0]{\tilde{Q}}
\newcommand{\Stil}[0]{\tilde{S}}

\newcommand{\Lcaltil}[0]{\tilde{\Lcal}}
\newcommand{\Pcaltil}[0]{\tilde{\Pcal}}
\newcommand{\Qcaltil}[0]{\tilde{\Qcal}}

\newcommand{\elltil}[0]{\tilde{\ell}}
\newcommand{\norm}[1]{\ensuremath{\|#1\|}}

\newcommand{\tel}[1]{\ensuremath{ |#1| }}

 \newcommand{\eps}{\varepsilon}
\DeclareMathOperator{\bsize}{bitsize}
\DeclareMathOperator{\conv}{conv}

\DeclareMathOperator{\spanH}{span}
\DeclareMathOperator{\clo}{cl}
\DeclareMathOperator{\cross}{cross}

\newcommand{\constr}[1]{{\bf(CPC-#1)}}
\newcommand{\Lcon}[1]{{\bf(\Lcal-#1)}}
\newcommand{\SSS}[1]{{\bf(S-#1)}}
\newcommand{\DG}[0]{\ensuremath{{\cal{DG}}}}
\newcommand{\UDG}[0]{\ensuremath{{\cal{UDG}}}}
\newcommand{\SEG}[0]{\ensuremath{{\cal{SEG}}}}

\newcommand{\fDG}[0]{\ensuremath{f_{\cal{DG}}}}
\newcommand{\fUDG}[0]{\ensuremath{f_{\cal{UDG}}}}
\newcommand{\fSEG}[0]{\ensuremath{f_{\cal{SEG}}}}
\newcommand{\gDG}[0]{\ensuremath{g_{\cal{DG}}}}
\newcommand{\gUDG}[0]{\ensuremath{g_{\cal{UDG}}}}
\newcommand{\gSEG}[0]{\ensuremath{g_{\cal{SEG}}}}

 \title{Integer realizations of disk and segment graphs}
  \author{Colin McDiarmid \\ University of Oxford \\ {\tt cmcd@stats.ox.ac.uk}
\and Tobias M\"uller\thanks{Research partially supported by a VENI grant from Netherlands Organisation for Scientific Research (NWO).} \\ Centrum Wiskunde \& Informatica \\ {\tt tobias@cwi.nl}
}
  
\begin{document}

  \maketitle

\begin{abstract}
A disk graph is the intersection graph of disks in the plane, a unit disk graph is
the intersection graph of same radius disks in the plane, and a segment graph
is an intersection graph of line segments in the plane.
It can be seen that every disk graph can be realized by disks with centers on the integer 
grid and with integer radii; and similarly every unit disk graph can be realized by disks 
with centers on the integer grid and equal (integer) radius; and every segment graph
can be realized by segments whose endpoints lie on the integer grid. 
Here we show that there exist disk graphs on $n$ vertices such that
in every realization by integer disks at least one coordinate or 
radius is $2^{2^{\Omega(n)}}$ and on the other hand 
every disk graph can be realized by disks with integer coordinates and radii
that are at most $2^{2^{O(n)}}$; and we show the
analogous results for unit disk graphs and segment graphs.
For (unit) disk graphs this answers a question of Spinrad, and
for segment graphs this improves over a previous result
by Kratochv\'{\i}l and Matou{\v{s}}ek.
\end{abstract}

%%%%%%%%%%%%%%%%%%%%%%%%%%%%%%%%%%%%%%%%%%%%%%%%%%%%%%%%%%%%%%%%%%%%%%%%%%%%%%%
%%%%%%%%%%%%%%%%%%%%%%%%%%%%%%%%%%%%%%%%%%%%%%%%%%%%%%%%%%%%%%%%%%%%%%%%%%%%%%%

\section{Introduction and statement of results}

In this paper we will consider intersection graphs of disks and segments in the plane.
Over the past 20 years or so, intersection graphs of geometric objects in the plane, especially 
unit disk graphs (intersection graphs of equal radius disks), have been considered 
by many different authors. 
Partly because of their relevance for practical applications 
one of the main foci is the design of (efficient) algorithms for them. 
One can of course store a (unit) disk graph in a computer as an adjacency
matrix or a list of edges, but for many purposes it might be useful to
actually store a geometric representation in the form of the radii and the coordinates of the centres of the 
disks.
It can be seen that every disk graph can be realized by disks with centers on the integer 
grid $\Zed^2$ and with integer radii; and similarly every unit disk graph can be realized by disks 
with centers on the integer grid and equal (integer) radius.
A formal proof is given in section~\ref{sec:ub} below.
Spinrad~\cite{SpinradBoek} asked whether every (unit) disk graph 
has a such a representation where all the integers
involved are at most $2^{O(n^K)}$ for some fixed $K$.
This question was also studied by Van Leeuwen and Van Leeuwen~\cite{VanleeuwenVanleeuwenTechreport}
who called it the {\em Polynomial Representation Hypothesis} (PRH). They showed 
that a positive answer is equivalent to a statement about spacings between the points of
a realization and they discussed some of its implications.
(See also chapter 4 of~\cite{erikjanthesis}.)
Theorems~\ref{thm:DG} and~\ref{thm:UDG} below show that the
PRH in fact fails as there are (unit) disk graphs that 
require integers that are doubly exponentially large in the number of vertices $n$.
Theorem~\ref{thm:SEG} establishes the analogous results for segment graphs if 
we place all the endpoints of the segments on $\Zed^2$.
This improves over earlier work of Kratochv{\'{\i}}l and Matou{\v{s}}ek~\cite{KratochvilMatousek94}
who had showed that integers that are doubly exponentially large in the
{\em square root} of the number of vertices may be needed for segment graphs.

Breu and Kirkpatrick~\cite{BreuKirkpatrick98} proved that the algorithmic decision problem of recognizing 
unit disk graphs (i.e.~given a graph in adjacency matrix form as input, decide whether
it is a unit disk graph) is NP-hard; and Hlin\v{e}n\'y and Kratochv{\'{\i}}l~\cite{HlinenyKratochvil2001} proved that recognizing disk graphs is NP-hard.
%%%%, and Kratochv{\'{\i}}l and Matou{\v{s}}ek~\cite{KratochvilMatousek94} proved that recognizing segment graphs is NP-hard. 
Had it been true, the PRH would have proved that
these problems are also members of the complexity class NP.
%%A natural question is thus whether these 
%%%decision problems are also members of NP.
For a decision problem to be in NP we need a ``polynomial certificate'';
that is, for each graph that is a (unit) disk graph there should be a proof 
of this fact in the form of a polynomial size bit string 
that can be checked by an algorithm in polynomial time. 
An obvious candidate for such a certificate is a list of the radii and the coordinates of 
the centres of the disks representing the graph.
For this to be a good certificate, we would however need to guarantee
that these coordinates and radii can be stored using polynomially many bits
(which would be the case if they were all integers bounded by $2^{O(n^K)}$).
We will also see that if we use rational coordinates and radii 
rather than integers %%%%(and we store these rationals in the standard way) 
then we still need an exponential number of bits.

We will now proceed to give the necessary definitions to 
state our results more formally.

% Spinrad~\cite{SpinradBoek} also asked how many bits are needed to 
% store the coordinates of a $k$-dot product graph, motivated by 
% the so-called {\em implicit graph conjecture} of Kannan, Naor and Rudich~\cite{KannanEtal92}.
% This conjecture asserts that, for every 
% hereditary graph class $\Ccal$ (i.e.~$\Ccal$ is closed under taking induced subgraphs)
% such that the number of graphs on $n$ vertices in $\Ccal$ is
% $2^{O(n \ln n)}$, there exists a ``labelling scheme'' that encodes 
% graphs in the class by assigning bit strings of length $O(\ln n)$ to the 
% vertices in such a way that the adjacency of two vertices $u,v$ can be 
% tested by examining only the labels of $u$ and $v$.
% It is easily shown that the number of $k$-sphere graphs (resp.~$k$-dot product graphs) 
% on $n$ vertices is $2^{O(n \ln n)}$ by Warren's Theorem, cf.~\cite{SpinradBoek}.
% A natural candidate for the labelling scheme would be to use the vectors $v_1,\dots,v_n$ of some $k$-sphere (resp.~$k$-dot product) representation 
% as the labels.
% For this to work, we would however need to guarantee that 
% each of $v_1,\dots, v_n$ can be stored using only $O( \ln n )$ bits.

\subsection{Statement of results}

If $\Acal = ( A(v) : v \in V )$ is a tuple of sets, then the {\em intersection graph}
of $\Acal$ is the graph $G = (V,E)$ with vertex set $V$, and an edge $uv \in
E$ if and only if $A(v) \cap A(v) \neq \emptyset$. 
We say that $\Acal$ {\em realizes} $G$.
If all the sets $A(v)$ are closed line segments in the plane then we speak of 
a {\em segment graph}.
A {\em disk graph} is an intersection graph of open disks in the plane, and 
if all the disks can be taken to have the same radius, then we speak of
a {\em unit disk graph}.
Let us denote by $\DG$ the set of all graphs that are (isomorphic to) disk graphs; and similarly
let $\UDG$ resp.~$\SEG$ denote the set of all graphs that are (isomorphic to) unit disk graphs 
resp.~segment graphs. Let us also set 
$\DG(n) := \{ G \in \DG : |V(G)| = n \}, 
\UDG(n) := \{ G \in \UDG : |V(G)| = n \},
\SEG(n) := \{ G \in \SEG : |V(G)| = n \}$.

If $\Acal = (A(v):v\in V(G))$ is a realization of $G\in\DG$ and
all the $A(v)$ are disks then we also say that 
$\Acal$ is a {\em$\DG$-realization} of $G$ or a realization of $G$ as a 
disk graph.
Similarly if the $A(v)$ are disks of the same radius, then we say 
$\Acal$ is a {\em$\UDG$-realization} of $G$; and 
if all the $A(v)$ are segments then we say that $\Acal$ is a
{\em$\SEG$-realization} of $G$.

It can be seen that every $G\in\DG$ has a $\DG$-realization in which the centers
of the disks $A(v)$ lie on $\Zed^2$ and the radii are integers. 
(A formal proof is given in section~\ref{sec:ub} below.)
We shall call such a realization an {\em integer $\DG$-realization} 
or just an integer realization if no confusion can arise.
Similarly in an {\em integer $\UDG$-realization} of $G\in\UDG$ all 
the disks $A(v)$ have centers $\in \Zed^d$ and a common radius $\in\eN$;
and in an {\em integer $\SEG$-realization} of $G\in\SEG$ both endpoints
of each segment $A(v)$ lie on $\Zed^2$. 

If $\Acal$ is a collection of bounded sets in the plane, then we will denote:

\[ 
k(\Acal) := \min\{ k \in \eN : A \subseteq [-k,k]^2 \text{ for all } A \in \Acal \}.
\]

\noindent
For $G\in\DG$ we will denote 

\[ 
\fDG(G) = \min_{\Acal \text{ integer $\DG$-realization}\atop \text{ of $G$} }
k(\Acal), \]

\noindent
and let us define $\fUDG(G)$ for $G \in \UDG$ and
$\fSEG(G)$ for $G\in\SEG$ analogously.
We now define

\[
\fDG(n) := \max_{G \in \DG(n)} \fDG(G), 
\]

\noindent
and we define $\fUDG(n)$ and $\fSEG(n)$ analogously.
Phrased differently, $\fDG(n)$ tells us the smallest square piece of the integer grid on which we can realize 
any disk graph on $n$ vertices.

Both Spinrad~\cite{SpinradBoek} and Van Leeuwen and Van Leeuwen~\cite{VanleeuwenVanleeuwenTechreport}
(see also chapter 4 of~\cite{erikjanthesis}) asked whether $\fDG(n)$ and
$\fUDG(n)$ are bounded by $2^{O(n^K)}$ for some constant $K$.
This was called the {\em polynomial representation hypothesis} (for disk/unit disk graphs) 
by Van Leeuwen and Van Leeuwen.
Here we will disprove the polynomial representation hypothesis:

\begin{theorem}\label{thm:DG}
$\fDG(n) = 2^{2^{\Theta(n)}}$.
\end{theorem}

\begin{theorem}\label{thm:UDG}
$\fUDG(n) = 2^{2^{\Theta(n)}}$.
\end{theorem}

\noindent
Theorem~\ref{thm:UDG} also improves over the conference version~\cite{McDiarmidMuller10} of this work, where
we proved a lower bound of $\fUDG(n) = 2^{2^{\Omega(\sqrt{n})}}$.

As it happens, $\fSEG(n)$ has been previously studied by Kratochv{\'{\i}}l and Matou{\v{s}}ek~\cite{KratochvilMatousek94},
who gave a lower bound of $\fSEG(n) \geq 2^{2^{\Omega(\sqrt{n})}}$.
Here we will improve over their lower bound, and give a matching upper bound:

\begin{theorem}\label{thm:SEG}
$\fSEG(n) = 2^{2^{\Theta(n)}}$.
\end{theorem}

% One can of course store a (unit) disk graph in a computer as an adjacency matrix
% or a list of edges, but for many purposes (algorithms) it is useful to actually store a
% geometric representation.
% In this article we will study in the number of bits that are needed to store such a representation.
% If $G$ is a disk graph then a {\em realization}
% will be a vector $\Rcal = (x_1,y_1,r_1, \dots, x_n, y_n, r_n) \in \Qu^{3n}$ 
% that stores the coordinates and radii of disk $D_1, \dots, D_n$ in the plane such that 
% $G$ is their intersection graph.
% There are of course infinitely many realizations, but we will focus on a 
% realization whose coordinates have the smallest possible bit size.

A standard convention (see~\cite{schrijverboek}) is to store rational numbers in the 
memory of a computer as a pair of integers (the denominator and numerator) that are relatively prime 
and those integers are stored in the binary number format.
The bit size of an integer $n \in \Zed$ is 

\[ 
\bsize(n) := 1 + \lceil\log_2(|n|+1)\rceil, 
\]

\noindent 
(the extra one is for the sign) and the bit size of a rational number $q \in \Qu$ is
$\bsize(q) = \bsize(n)+\bsize(m)$ if $q=\frac{n}{m}$ and $n, m$ are
relatively prime.
For $G \in \DG$, let $\gDG(G)$ denote the 
minimum, over all realizations by disks with centers
in $\Qu^2$ and rational radii, of the sum of the bit sizes of the
coordinates of the centers and the radii; and
let $\gDG(n)$ denote the maximum of
$\gDG(G)$ over all $G \in \DG(n)$.
Let us define $\gUDG(n)$ and $\gSEG(n)$ analogously.
By definition of $\fDG(n)$, every 
$G \in \DG(n)$ has a rational realization where each of the $3n$ 
coordinates and radii has numerator
at most $\fDG(n)$ in absolute value and denominator equal to 1.
Hence

\begin{equation}\label{eq:bsize1}
\gDG(n) \leq 3n \cdot \bsize(\fDG(n)) + 6n.
\end{equation}

\noindent
If we multiply all the coordinates and radii 
of a rational realization by the product of their denominators 
we get an integer realization.
In the resulting integer realization each coordinate or radius will have a bit size
that does not exceed the sum of the bit sizes of the
coordinates and radii of the original rational realization. 
(Observe that for two integers $n,m$ we have $\bsize(nm) \leq \bsize(n)+\bsize(m)$.)
Thus:

\begin{equation}\label{eq:bsize2}
\bsize(\fDG(n)) \leq \gDG(n), 
\end{equation}

\noindent
%%%using that for two integers $n,m$ we have $\bsize(nm) \leq \bsize(n)+\bsize(m)$.
From the definition of $\bsize(.)$ and Theorem~\ref{thm:DG} we have
$\bsize(\fDG(n)) = 2^{\Theta(n)}$.
Combining this with~\eqref{eq:bsize1} and~\eqref{eq:bsize2} we find:

\begin{corollary}
$\gDG(n) = 2^{\Theta(n)}$. \noproof
\end{corollary}

\noindent
Similarly we have

\begin{corollary}
$\gUDG(n) = 2^{\Theta(n)}$.\noproof
\end{corollary}

\begin{corollary} 
$\gSEG(n) = 2^{\Theta(n)}$. \noproof
\end{corollary}

% Our main result in this paper is the following:

% \begin{theorem}\label{thm:lb} 
% For all sufficiently large $n$, there
% exists a disk graph on $n$ vertices with $\bsize(G) > {(1.38)}^{\sqrt{n}}$.
% \end{theorem}

% Theorem~\ref{thm:lb} could be seen as bad news for those wishing to design algorithms
% for unit disk graphs. On the slightly positive side we offer the following upper
% bound:

% \begin{theorem}\label{thm:ub}
% There exists a constant $\gamma$ such that
% for each $n$, each unit disk graph $G$ on $n$ vertices
% has $\bsize(G) \leq \gamma^{n}$.
% \end{theorem}

% Our results also hold for disk graphs (intersection graphs of disks not all of the same radius), but the 
% proofs are more involved.
% We therefore postpone these proofs to the journal version of this paper.

Before beginning our proofs in earnest we will give a brief overview of the structure of the paper 
and the main ideas in the proofs.

\subsection{Overview of the paper and sketch of the main ideas in the proofs}

The proofs of the doubly exponential upper bounds in Theorems~\ref{thm:DG},~\ref{thm:UDG}
and~\ref{thm:SEG} can be found in section~\ref{sec:ub}.
The upper bound of Theorem~\ref{thm:SEG} is a direct consequence of a
result of Goodman, Pollack and Sturmfels~\cite{GPS90}, and the
upper bounds in Theorems~\ref{thm:DG} and~\ref{thm:UDG} are both relatively
straightforward consequences of a result of Grigor'ev and Vorobjov~\cite{grigorevvorobjov88}.

The proofs of the lower bounds are substantially more involved.
The main ingredient is Theorem~\ref{thm:sizeLline}, which gives a
construction of an oriented line arrangement $\Lcal$ and a set of points
$\Pcal$ 
with the property that whenever some set of points $\Pcaltil$ have the same sign vectors
with respect to an oriented line arrangement $\Lcaltil$ as the points $\Pcal$ have with respect to $\Lcal$
%%(here we mean by ``in corresponding position" that they lie in corresponding cells of the 
%%dissection of the plane $\eR^2$ induced by $\Lcal$)
then there are two pairs of intersection points of the lines
$\Lcaltil$ such that the distance between one pair is a factor $2^{2^{\Omega(|\Lcal|)}}$ larger than 
the distance between the other pair.
An important property of the construction is that the number of points $|\Pcal|$ is linear in the
number of lines $|\Lcal|$. 
This ultimately allows us to give lower bounds that are doubly exponential in the number of
vertices $n$ rather than the {\em square root} $\sqrt{n}$. %%% of the number of vertices.
(And thus it allows us to get sharp lower bounds, and improve over our own result for 
unit disk graphs in the conference version~\cite{McDiarmidMuller10} of this paper, and the lower bound of 
Kratochv{\'{\i}}l and Matou{\v{s}}ek~\cite{KratochvilMatousek94} for segment graphs.)
A brief sketch of the construction of $\Lcal$ and $\Pcal$ is as follows.
We start with a constructible point configuration $\Qcal$ such that, in 
every point configuration $\Qtil$ that is projectively equivalent to it, there are four points 
with a large cross ratio. (We construct such a constructible point configuration using Von Staudt sequences, a 
classical way to encode arithmetic operations into point configurations.)
We now construct $\Lcal, \Pcal$ step-by-step by adding four lines and eleven points
for each $q\in\Qcal$.
These four lines and eleven points per point of $\Qcal$ are placed in 
such a way (inspired by constructions of Shor~\cite{Shor91} and Jaggi et al.~\cite{JaggiEtAl89}) that 
in any $\Lcaltil$ for which some point set $\Pcaltil$ has the same sign vectors
with respect to $\Lcaltil$ as $\Pcal$ has with respect to $\Lcal$ we can construct a
point configuration $\Qcaltil$ that is projectively equivalent 
to $\Qcal$ and whose points lie in prescribed cells of $\Lcaltil$.
%%The way in which we place these four lines and eleven points is inspired by
%%%constructions of Shor~\cite{Shor91} and Jaggi et al.~\cite{JaggiEtAl89}.

With our main tool at hand we can construct unit disk graphs, disk graphs, and segment
graphs that need large parts of the integer grid to be realized.
For unit disk graphs the idea is to take the oriented line arrangement from Theorem~\ref{thm:sizeLline}
and construct a unit disk graph with a pair of vertices for each line of $\Lcal$, one corresponding to each 
halfplane defined by the line, and one vertex for each point of $\Pcal$.
The constructed unit disk graph is such that, in every realization of it, we can construct
a line arrangement $\Lcaltil$ and a set of points $\Pcaltil$ out of the coordinates
of the centers of the disks such that the position of $\Pcaltil$ with respect to 
$\Lcaltil$ is the same as the position of $\Pcal$ with respect to $\Lcal$.
It then follows from some calculations that at least one coordinate or radius is
doubly exponentially large.

The construction for disk graphs is very similar, except that now
we are forced to place an induced copy of a certain graph $H$ for each point in $\Pcal$ rather than a 
single vertex.
The graph $H$ has the property that for every realization $\Dcal$ of it there is
a point $p(\Dcal)$ such that any disk that intersects all disks of $\Dcal$ contains
$p(\Dcal)$.
The construction of $H$ is pretty involved.
It relies heavily on the fact (Lemma~\ref{lem:trianglefreeplanar} below) that in every
realization of a triangle-free disk graph $G$ with minimum degree
at least two, the centers of the disks define a straight-line drawing of $G$.
  
For segment graphs the construction of a segment graph that needs a large portion of the integer 
grid makes use of a convenient result of Kratochv{\'{\i}}l and Matou{\v{s}}ek~\cite{KratochvilMatousek94}, 
the ``order forcing lemma".
This time we use one vertex per line of $\Lcal$, and two per point of
$\Pcal$, and a constant number of vertices that are needed to set the construction up in such a way that we
can apply the order forcing lemma.

In section~\ref{sec:CPC} we do some preliminary work
needed for the proof of Lemma~\ref{thm:sizeLline}
in section~\ref{sec:OLA}.
The material in section~\ref{sec:CPC} is classical and may be 
skimmed by readers familiar with constructible point configurations
and Von Staudt sequences.
In section~\ref{sec:OLA} we prove Theorem~\ref{thm:sizeLline}, our main tool.
% The proof transforms a constructible point configuration with a large cross
% ratio, constructed in section~\ref{sec:CPC}, into a line arrangement $\Lcal$ and 
% point set $\Pcal$ with $|\Pcal|=O(|\Lcal|)$ and other desirable properties.
% The proof of Theorem~\ref{thm:sizeLline} uses constructions inspired by 
% constructions pf Shor~\cite{Shor91} and Jaggi et al.~\cite{JaggiEtAl89}.
%
Section~\ref{sec:UDGlb} contains the lower bound for unit disk graphs, section~\ref{sec:DGlb} 
contains the lower bound for disk graphs and section~\ref{sec:SEGlb} has the
lower bound for segment graphs.
As mentioned earlier, upper bounds are proved in Section~\ref{sec:ub}.

\section{Constructible point configurations\label{sec:CPC}}

Although we are mainly interested in intersection graphs of objects in the (ordinary) euclidean plane it is 
convenient to do some preliminary work in the projective setting.
Recall that the {\em real projective plane} $\RP^2$ has as its
points the one-dimensional linear subspaces of $\eR^3$, and as its lines the 
two-dimensional linear subspaces of $\eR^3$. 
We denote a point of $\RP^2$ in homogeneous coordinates
as $(x:y:z) := \{ (\lambda x,\lambda y,\lambda z)^T : \lambda\in\eR \}$ -- where $(x,y,x)\neq (0,0,0)$.
We say that $v \in \eR^3 \setminus \{0\}$ is a {\em representative} of 
$p = (x:y:z) \in \RP^2$ if $v \in p$.
%%%(Recall that $(x:y:z) = (\lambda x:\lambda y:\lambda z)$ for $\lambda\neq 0$.
The euclidean plane $\eR^2$ is contained the projective plane
via the {\em canonical embedding} $(x,y)^T \mapsto (x:y:1)$.
The points of $\RP^2$ that do not lie on $\eR^2$
are all points of the form $(x:y:0)$, and they form 
a line of $\RP^2$ (they correspond to the plane $\{z=0\}$ in $\eR^3$), called 
the {\em line at infinity}.
A convenient property of the projective plane is that {\em every two lines meet in a point}.
If two lines are parallel in the euclidean plane, then they have an intersection point
on the line at infinity in the projective plane.
A {\em projective transformation} is the action that a non-singular linear map
$T:\eR^3\to\eR^3$ induces on $\RP^2$.
(Observe that it sends the points of $\RP^2$ to points of $\RP^2$ and the
lines of $\RP^2$ to lines of $\RP^2$.)
Recall that an {\em isometry} of the euclidean plane is a map that preserves distance, and 
that an isometry can always be written as a translation followed by the
action of an orthogonal linear map.
We omit the straightforward proof of the following
observation.

\begin{lemma}\label{lem:isometry}
If $f:\eR^2\to\eR^2$ is an isometry then there exists
a projective transformation $T:\RP^2\to\RP^2$ such that
the restriction of $T$ to $\eR^2$ coincides with $f$. \noproof
%%%$T\upharpoonright\eR^2 = f$. \noproof
\end{lemma}

% \begin{proof}
% Since the composition of two projective transformations is again a projective transformation, 
% it suffices to do only the cases when $f:\eR^2\to\eR^2$ is either a translation or 
% a linear transformation.
% First suppose $f$ is a translation, say $f((x,y)^T) = (x,y)^T + (a,b)^T$.
% Let $T:\eR^3\to\eR^3$ be the unique linear map that satisfies
% $e_1 \mapsto e_1, e_2\mapsto e_2, e_3\mapsto(a,b,1)$.
% Then we clearly have $T((x,y,1)^T) = (x+a,y+b,1)^T$ for all $x,y$ and
% hence also

% \[ T[(x:y:1)] = ((x+a):(y+b):1). \]

% \noindent
% If $f:\eR^2\to\eR^2$ is a linear transformation, then 
% $T:\eR^3\to\eR^3$ given by $(x,y,z)^T \mapsto (f(x,y),z)$ 
% clearly does the trick.    
% \end{proof}

\noindent
For vectors $u,v,w \in \eR^3$ we will write 

\[ 
[u,v,w] := \det\left(
\begin{array}{ccc}
u_x & u_y & u_z \\
v_x & v_y & v_z \\
w_x & w_y & w_z 
\end{array}
\right).
\]

\noindent
If $a,b,c,d \in \RP^2$ are four collinear points, and
$p \in \RP^2$ is a point not on the line spanned by them, then the
{\em cross ratio} can be defined as:

\begin{equation}\label{eq:crossdef}
\cross(a,b,c,d) := 
\frac{[p,a,c][p,b,d]}{[p,a,d][p,b,c]}.
\end{equation}

\noindent
Here we take arbitrary representatives of $a,b,c,d,p$ in the right-hand side.
(That is, if $a=(a_x:a_y:a_z)$ then we may take $(\lambda a_x,\lambda a_y, \lambda a_z)$ for
any $\lambda \in \eR \setminus \{0\}$, etc.)
To see that~\eqref{eq:crossdef} is a valid definition, recall that the determinant is linear in
each of its rows (meaning that 
$[u_1+u_2,v,w] = [u_1,v,w] + [u_2,v,w], [\lambda u,v,w] = \lambda [u,v,w]$ etc.).
Thus, if instead of $(a_x,a_y,a_z)$ we take $(\lambda a_x,\lambda a_y, \lambda a_z)$ in~\eqref{eq:crossdef}
then we just get a factor of $\lambda$ in both the denominator and the numerator. Similarly for $b,c,d$.
Let $\ell$ denote the line that $a,b,c,d$ are on, and let $H \subseteq \eR^3$ denote 
the corresponding two-dimensional linear subspace.
To see that the choice of $p$ does not matter, let us fix a $u \in \eR^3 \setminus \{0\}$ that is orthogonal to $H$.

Pick $p \in \RP^2 \setminus \ell$, and let $z$ be an arbitrary 
representative of $p$. So in other words $z \in \eR^3 \setminus H$. 
We can write $z = u_1 + \lambda u$ with $u_1 \in H$.
Then we get $[z,a,c] = [u_1,a,c] + \lambda [u,a,c] = \lambda[u,a,c]$, and similarly
$[z,a,d] = \lambda[u,a,d], [z,b,c] = \lambda[u,b,c], [z,b,d] = \lambda[u,b,d]$.
We see that

\[ 
\frac{[z,a,c][z,b,d]}{[z,a,d][z,b,c]} 
= 
\frac{[u,a,c][u,b,d]}{[u,a,d][u,b,c]}, 
\]

\noindent
no matter which $p \in \RP^2 \setminus \ell$ resp.~$z\in \eR^3 \setminus H$ we start from.
So the definition of $\cross(a,b,c,d)$ in~\eqref{eq:crossdef} does indeed not depend
on the choice of representatives for $a,b,c,d$ or the choice of $p\in\RP^2 \setminus \ell$.
 
Next, let us also remark that, since $\det(AB) = \det(A)\det(B)$ for square matrices $A,B$, we have
$[Tu,Tv,Tw] = \det(T)[u,v,w]$ for any projective transformation $T$.
This implies that the cross ratio is preserved under projective transformations:

\begin{lemma}\label{lem:crossproj}
If $T$ is a projective transformation and $a,b,c,d \in \RP^2$ are collinear, then
$\cross(a,b,c,d) = \cross(Ta,Tb,Tc,Td)$.\noproof
\end{lemma}

\noindent
In the projective plane there is no obvious way to define a notion of distance, but
for collinear points in the euclidean plane the cross ratio can be 
related to euclidean distances:

\begin{lemma}\label{lem:crosseucl}
If $a,b,c,d \in \eR^2$ are collinear, then 

\[ |\cross(a,b,c,d)| = 
\frac{\norm{a-c}\cdot\norm{b-d}}{\norm{a-d}\cdot\norm{b-c}}. \]

\end{lemma}

\begin{proof}
If we apply an isometry that maps the line that contains $a,b,c,d$ 
to the $x$-axis, then the cross-ratio does not change by Lemma~\ref{lem:isometry} and 
Lemma~\ref{lem:crossproj}.
Hence we can assume $a = (\alpha,0)^T,b=(\beta,0)^T,c=(\gamma,0)^T,d=(\delta,0)^T$
for some $\alpha,\beta,\gamma,\delta\in\eR$.
Or, in projective terms we have $a = (\alpha:0:1),
b=(\beta:0:1), c=(\gamma:0:1), d=(\delta:0:1)$.
Taking $p = (1:1:1)$ we see that 

\begin{equation}\label{eq:acdet}
[p,a,c] = \det\left( \begin{array}{ccc}
1&1&1 \\ \alpha&0&1\\ \gamma&0&1 \end{array}\right)
= \gamma-\alpha = \pm\norm{a-c},
\end{equation}

\noindent
and similarly $[p,a,d] = \pm\norm{a-d}, [p,b,c]=\pm\norm{b-c}, 
[p,b,d] = \pm\norm{b-d}$. This proves the lemma.
\end{proof}

A {\em point configuration} is a tuple $\Pcal = (p_1,\dots,p_n)$ of (labelled) points 
in the projective plane.
%%%We will always assume that these points are distinct.
If all the points lie in the euclidean plane $\eR^2$ then we speak of a {\em euclidean point configuration}.
For two distinct points $p,q$ we shall denote by $\ell(p,q)$ the unique line through $p$ and $q$.
We call a point configuration {\em constructible} if

\begin{enumerate}
\item[\constr{1}] No three of $p_1,p_2,p_3,p_4$ are collinear;
\item[\constr{2}] For each $i\geq 5$ there are $j_1,j_2,j_3,j_4 < i$ such that 
$\{p_i\} = \ell(p_{j_1},p_{j_2})\cap\ell(p_{j_3},p_{j_4})$.
\end{enumerate}

\noindent
We will say that two point configurations 
$\Pcal = (p_1,\dots, p_n), \Pcaltil = (\ptil_1,\dots,\ptil_n)$ 
are {\em projectively equivalent} if there exists a projective
transformation $T$ such that $\ptil_i = Tp_i$ for all $i=1,\dots,n$.
Observe that if $T$ is a projective transformation, then 
$T[\ell(p,q)] = \ell(Tp,Tq)$. Thus:

\begin{lemma}\label{lem:constrproj}
If $\Pcal$ is a constructible point configuration and 
$\Pcaltil$ is projectively equivalent to $\Pcal$ then
$\Pcaltil$ is also constructible.\noproof
\end{lemma}

\noindent
Recall that we say a configuration of points is {\em in general position} if 
no three of them are collinear.
The following observation will be needed in the next section.

\begin{lemma}\label{lem:4ptsprojeq}
If $\Pcal = (p_1,p_2,p_3,p_4)$ and $\Pcaltil = (\ptil_1,\ptil_2,\ptil_3,\ptil_4)$
are two point configurations on 4 points in general position, then
$\Pcal$ and $\Pcaltil$ are projectively equivalent. 
\end{lemma}

\begin{proof}
Since the inverse of a projective transformation is again a projective transformation and
the composition of two projective transformations is a projective transformation, it
suffices to prove the result for
$p_1 = (1:0:0), p_2=(0:1:0), p_3=(0:0:1), p_4 = (1:1:1)$
and $\ptil_1,\ptil_2,\ptil_3,\ptil_4$ arbitrary (but in general position).
Let us pick representatives $v_i \in\eR^3\setminus\{0\}$ of $\ptil_i$ for $i=1,\dots,4$.
Then no $v_i$ is a linear combination of only two of the other $v$s.
Hence there are nonzero $\lambda_i$s such that:

\[ 
v_4 = \lambda_1v_1+\lambda_2v_2+\lambda_3v_3. 
\]

\noindent
Let us now define a linear map
$T:\eR^3\to\eR^3$ by setting 
$e_i\mapsto\lambda_iv_i$ for $i=1,\dots,3$.
Then it is easy to see that, when viewed as a projective transformation, $T$ in fact maps $p_i$ to $\ptil_i$ for
$i=1,\dots,4$.
\end{proof}

\noindent
The so-called {\em Von Staudt sequences}, originally invented by Von Staudt~\cite{VonStaudt1847}
in 1847, allow us to encode arithmetic operations in terms of constructible point configurations.
Let us write

\begin{equation}\label{eq:P0PinfQRdef}
\begin{array}{l}
P_0 := (0:0:1), \quad P_{\infty} := (1:0:0), \\
Q := (0:1:0), \quad R:= (1:1:1).
\end{array}
\end{equation}

\noindent
And, for $a \in \eR$ let us set

\begin{equation}\label{eq:Pxdef}
P_a := (a:0:1).
\end{equation}

\noindent
The idea of the Von Staudt sequences is that, given a
point configuration that contains $P_0,P_{\infty},R,Q$
we can ``construct" the point $P_1$ as an intersection point (as in~\constr{2} above), and
if the point configuration also contains $P_a, P_b$ for some $a,b\in\eR$ then we can ``construct" the 
point $P_{a+b}$ resp.~$P_{a\cdot b}$
by defining a number of intersection points (as in \constr{2} above) the last of which will be 
$P_{a+b}$ resp.~$P_{a\cdot b}$.
There are also Von Staudt sequences for subtraction and division, but we shall not need them here.

The Von Staudt sequence for one is:

\begin{equation}\label{eq:one}
\{O_1\} := \ell(R,Q) \cap \ell(P_0,P_{\infty}).
\end{equation}

\noindent
See figure~\ref{fig:VonStaudt}, left.
The Von Staudt sequence for addition is to set (in this order):

\begin{equation}\label{eq:addition}
\begin{array}{l}
\{A_1\} = \ell(R,P_{\infty})\cap\ell(P_a,Q), \\
\{A_2\} = \ell(P_0,A_1)\cap\ell(P_{\infty},Q), \\
\{A_3\} = \ell(P_b,A_2)\cap\ell(R,P_{\infty}), \\
\{A_4\} = \ell(A_3,Q)\cap\ell(P_0,P_{\infty}).
\end{array}
\end{equation}

\noindent
See figure~\ref{fig:VonStaudt}, middle.
The Von Staudt sequence for multiplication is to set (in this order):

\begin{equation}\label{eq:multiplication}
\begin{array}{l}
\{M_1\} = \ell(P_0,R)\cap\ell(P_b,Q), \\
\{M_2\} = \ell(P_a,R)\cap\ell(P_{\infty},Q), \\
\{M_3\} = \ell(M_1,M_2)\cap\ell(P_0,P_{\infty}). 
\end{array}
\end{equation}

\noindent
See figure~\ref{fig:VonStaudt}, right.

\begin{figure}[!ht]
\begin{center}
\hspace{-1.2cm}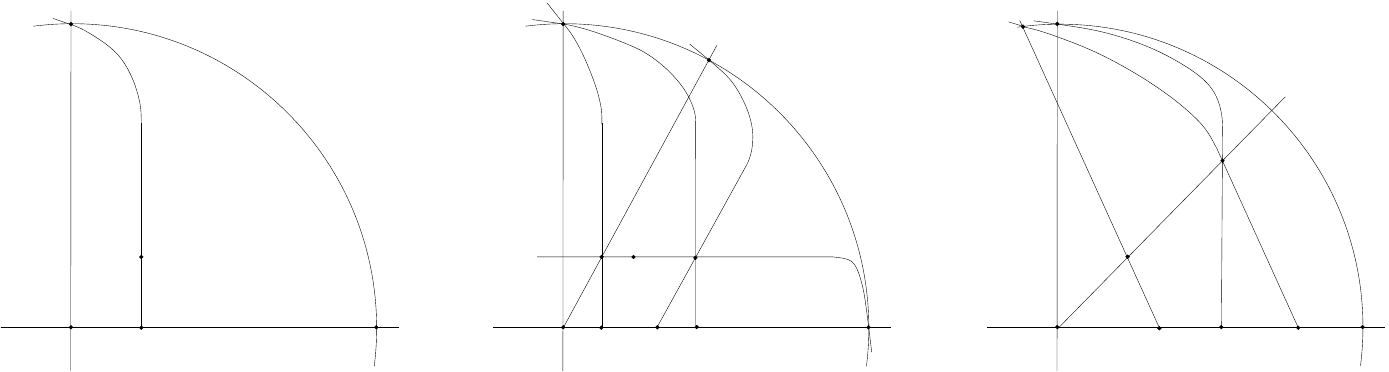
\end{center}
\caption{The Von Staudt constructions for one, left (where $O_1$ is shown as $P_1$), addition, middle
(where $A_4$ is shown as $P_{a+b}$) and multiplication, right
(where $M_3$ is shown as $P_{a\cdot b}$).
The line at infinity (the line through $P_{\infty}$ and $Q$) is drawn as a circular arc.\label{fig:VonStaudt}}
\end{figure}

\begin{lemma}\label{lem:VonStaudtCorrect}
Let $P_0,P_\infty,Q,R$ be as defined in~\eqref{eq:P0PinfQRdef} and
for arbitrary $a,b \in \eR$ let $P_a,P_b$ be as defined by~\eqref{eq:Pxdef}.
Then the following hold.
\begin{enumerate}
\item\label{itm:P1corr}
Let $O_1$ be as defined by~\eqref{eq:one}: 
then $O_1 = P_1$;
\item\label{itm:addcorr}
Let $A_1,\dots,A_4$ be as defined by~\eqref{eq:addition}:
then $A_4 = P_{a+b}$;
\item\label{itm:multcorr}
Let $M_1,M_2,M_3$ be as defined by~\eqref{eq:multiplication}:
then $M_3 = P_{a\cdot b}$.
\end{enumerate} 
\end{lemma}

\begin{proof}
It is convenient to consider what the Von Staudt sequences look like in the euclidean plane 
if we embed it in the projective plane via the canonical embedding
$(x,y)^T\mapsto (x:y:1)$.
The point $P_a$ then corresponds to 
$(a,0)^T$ for all $a\in\eR$.
Also observe that the $x$-axis corresponds to $\ell(P_0,P_\infty)$, and
the $y$-axis to $\ell(P_0,Q)$.
Two lines are parallel in the euclidean plane precisely if in the projective plane 
they intersect in a point on the line at infinity $\ell(P_{\infty},Q)$.
Thus horizontal lines in the euclidean plane are precisely 
the lines which intersect $\ell(P_{\infty},Q)$ in the point $P_{\infty}$, and
vertical lines are precisely the lines that intersect $\ell(P_{\infty},Q)$ in $Q$.

\noindent
{\bf Proof of~\ref{itm:P1corr}:}
In euclidean terms $R$ is the point $(1,1)^T$ and $\ell(R,Q)$ is the vertical
line through $R$, and $\ell(P_0,P_{\infty})$ is the $x$-axis.
(See figure~\ref{fig:VonStaudt}, left.)
Hence $O_1$, the intersection point of $\ell(R,Q)$ and $\ell(P_0,P_{\infty})$,
must corresponds to $(1,0)^T$.

\noindent
{\bf Proof of~\ref{itm:addcorr}:}
The points $A_1,A_3$ lie on the horizontal line $\ell(R,P_{\infty})$,
the lines $\ell(P_a,A_1),\ell(A_4,A_3)$ are vertical 
and the lines $\ell(P_0,A_1)$ and $\ell(P_b,A_3)$ are parallel.
Hence the triangle $P_bA_3A_4$ is a translate of the triangle $P_0A_1P_a$.
(See figure~\ref{fig:VonStaudt}, middle.)
In particular the segments $[P_0,P_a]$ and $[P_b,A_4]$ have the same length, and
so we must indeed have that $A_4$ coincides with the point $(a+b,0)^T$. 

\noindent
{\bf Proof of~\ref{itm:multcorr}:}
The line $\ell(P_0,R)$ coincides with the line $y=x$ in the euclidean plane.
Since the line $\ell(M_1,P_b)$ is vertical, the point 
$M_1$ corresponds to $(b,b)^T$.
Since the lines $\ell(P_a,R)$ and $\ell(M_3,M_1)$ are parallel, 
the triangles $P_0RP_a$ and $P_0M_1M_3$ are similar.
The height of $P_0RP_a$ is 1, and its base is $a$.
Since the height of $P_0M_1M_3$ is $b$, its base must be $ab$.
Hence $M_3$ coincides with $(ab, 0)^T$ as required.
\end{proof}

\begin{lemma}\label{lem:Pxcross} 
Let $P_0,P_{\infty},Q,R \in \RP^2$ be as defined by~\eqref{eq:P0PinfQRdef}, and
let $P_1,P_x \in \RP^2$ be as defined by~\eqref{eq:Pxdef}.
Then $\cross(P_1,P_x,P_{\infty},P_0) = x$.
\end{lemma}

\begin{proof}
Observe that for all $a,b\in\eR$:

\[
\det\left(\begin{array}{ccc}
1&1&1\\a&0&1\\b&0&1
\end{array}\right)
= b-a, \quad
\det\left(\begin{array}{ccc}
1&1&1\\a&0&1\\1&0&0
\end{array}\right)
= 1.
\]

\noindent
It is easily seen that $R$ does not lie on the line $\ell(P_0,P_{\infty})$.
Hence, by definition~\eqref{eq:crossdef} of the cross ratio:

\[
\cross(P_1,P_x,P_{\infty},P_0) = 
\frac{[R,P_1,P_{\infty}]\cdot[R,P_x,P_0]}{[R,P_1,P_0]\cdot[R,P_x,P_{\infty}]} 
= x, \]

\noindent
for all $x\in\eR$, as required.
\end{proof}

\noindent
We are now in a position to give a quick proof of 
the following lemma, which will play an important role in the next section.
The lemma is already proved implicitly in the beginning of the proof of 
Theorem 1 in the seminal work of Goodman, Pollack and Sturmfels~\cite{GPS90} 
(see also~\cite{GPS89conf}).
%%For completeness, and since the proof is cute and short, we give
%%our own construction here.
A similar construction was also invented independently by Kratochvil and 
Matou{\v{s}}ek in~\cite{KratochvilMatousekTechReport}, the technical report 
version of~\cite{KratochvilMatousek94}.

\begin{lemma}\label{lem:GPS}
For every $r \in \eN$, there exists a constructible euclidean point 
configuration $\Pcal = (p_1,\dots,p_n)$ on $n = 3r+6$ points
such that $p_1,p_2,p_5,p_n$ are collinear and  $\cross(p_5,p_n,p_2,p_1) = 2^{2^r}$.
\end{lemma}

\begin{proof} 
For any finite set $S \subseteq \RP^2$ 
we can find a projective transformation such that
$T[S] \subseteq \eR^2$ (if $T$ is the action of a matrix with i.i.d.~standard normal entries 
then $T$ will do the trick with probability one, for instance), so by Lemmas~\ref{lem:constrproj} and~\ref{lem:crossproj}
it suffices to define a suitable constructible point configuration in the projective plane.
Our four initial points will be $(p_1,p_2,p_3,p_4) = (P_0,P_{\infty},Q,R)$, and
we set $p_5 = P_1$.
Then we append the Von Staudt sequence for $P_{1+1}$, followed by 
the Von Staudt sequences for $P_{2\cdot 2}, P_{4\cdot 4}$ and so on until
$P_{2^{2^r}}$.
This gives a constructible point configuration on $n = 4+1+4+3(r-1) = 3r+6$ points,
and by Lemma~\ref{lem:Pxcross} we have
$\cross(p_5,p_n,p_2,p_1) = \cross(P_1,P_{2^{2^r}},P_{\infty},P_0) = 2^{2^r}$, as required.
\end{proof}

%%ie in general position (i.e.~no three points are collinear, and 
%%every line defined by two of them intersects every other line defined by them)
%%and $p_{i}$ is the intersection point 
%%of two lines defined by $p_1, \dots, p_{i-1}$ for $i \geq 5$.

%%While there is no natural definition of ``length'' in the projective plane,
%%there a natural way to extend the notion of cross ratio.

\section{Oriented line arrangements\label{sec:OLA}}

In this section all the action takes place exclusively in the euclidean plane again.
A line $\ell$ divides $\eR^2 \setminus \ell$ into two pieces.
In an {\em orientation} of $\ell$ we distinguish between these
two pieces by (arbitrarily) calling one of them $\ell^-$ the ``negative side" and
the other $\ell^+$ the ``positive side".
An {\em oriented line arrangement} is a tuple $\Lcal := (\ell_1, \dots, \ell_n)$ of 
distinct lines in the plane, each with an orientation.

%%%For convenience we shall only work with oriented line arrangements from now on.
The {\em sign vector} of a point $p \in \eR^2$ wrt.~an oriented 
line arrangement $\Lcal = (\ell_1,\dots,\ell_n)$ is 
the vector $\sigma(p;\Lcal) \in \{-,0,+\}^n$ defined as follows:

\[ (\sigma(p;\Lcal))_i = \left\{ \begin{array}{cl}
- & \text{ if } p \in \ell_i^{-}, \\
0 & \text{ if } p \in \ell_i, \\
+ & \text{ if } p \in \ell_i^{+}.
\end{array} \right. \]

\noindent
If $\Pcal \subseteq \eR^2$ is a set of points then we write

\[ \sigma(\Pcal;\Lcal) := \{ \sigma(p;\Lcal) : p \in\Pcal \}. \]

\noindent
The {\em combinatorial description} $\Dcal(\Lcal)$ of $\Lcal$ is the set of all sign vectors
$\Dcal(\Lcal) := \sigma(\eR^2;\Lcal)$.
The combinatorial description $\Dcal(\Lcal)$ is almost the same thing as the 
covectors of an oriented matroid and in fact it determines the oriented matroid
associated with $\Lcal$ (see~\cite{BjornerEtalBoek} for more details).
It should be mentioned that various other notions of a
combinatorial description of a line arrangement are in use
such as a local sequences, allowable sequences and wiring diagrams
(see for instance~\cite{FelsnerBoek}).
%%%It is however relatively straightforward to obtain one of the 
%%%other notions from the one we use and vice versa, and in fact 
%%%one can do this by means of an algorithm that runs in polynomial time.

Each connected component of $\eR^2 \setminus (\ell_1\cup\dots\cup\ell_n)$ is called 
a {\em cell} or a {\em region}.
%%%%%If a cell is unbounded we shall call it an {\em unbounded cell}.
All points in the same cell have the same sign vector, which does not
have 0 as a coordinate.
A sign vector with two or more zeroes corresponds to an intersection point 
of two or more lines, and a sign vector with exactly one zero corresponds to a
line segment. (See figure~\ref{fig:combidesc}.)

\begin{figure}[!ht]
\begin{center}
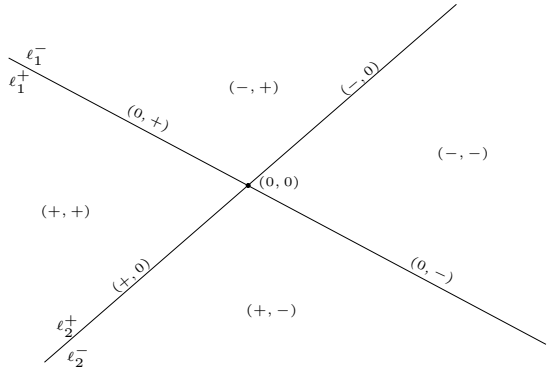
\end{center}
\caption{An oriented line arrangement and its sign vectors.\label{fig:combidesc}}
\end{figure}

\noindent
Moreover, let us observe that from the set of sign vectors $\Dcal(\Lcal)$ alone
we can determine all relevant combinatorial/topological information, such as whether
a given cell is a $k$-gon, which cells/segments/points are incident with
a given cell/segment/point, etc.
If $\Dcal(\Lcal) = \Dcal(\Lcal')$ then we say that $\Lcal$ and $\Lcal'$ are
{\em isomorphic}.
Informally speaking, isomorphic oriented line arrangements have the same ``combinatorial structure".

If every two lines of $\Lcal$ intersect, and no point is on more than two lines then
we say that $\Lcal$ is {\em simple}.
It can be seen that a simple oriented line arrangement has
exactly ${n+1 \choose 2}+1$ cells (for a proof of a generalization see for instance
Proposition 6.1.1~of~\cite{Matousekboek}).
We will need the following standard elementary observation:

\begin{lemma}\label{lem:simpleiso}
If $\Lcal$ is simple and $\Lcal'$ has the same number of lines, then 
$\Dcal(\Lcal) = \Dcal(\Lcal')$ if and only if 
$\{-,+\}^n \cap \Dcal(\Lcal) = \{-,+\}^n \cap \Dcal(\Lcal')$. \noproof
\end{lemma}

\noindent
(This is just the observation that in a simple oriented line arrangement we can 
reconstruct all other sign vectors from the nonzero ones.)

For $\Lcal$ an oriented line arrangement, let $\Ical(\Lcal)$ denote the set of {\em intersection
points}, that is all points $p \in \eR^2$ that lie on more than one line.
The {\em span} of an oriented line arrangement can be defined as 

\[ 
\spanH(\Lcal) := \frac{
\displaystyle \max_{p,q\in\Ical(\Lcal)} \norm{p-q} 
}{ 
\displaystyle 
\min_{p,q\in\Ical(\Lcal),\atop p\neq q} \norm{p-q} 
}.
\]

\noindent
(Thus $\spanH(\Lcal)$ is the ratio of the furthest distance between 
two intersection points to the smallest distance between 
two intersection points.)

The main tool in the proofs of the lower bounds 
in Theorems~\ref{thm:DG},~\ref{thm:UDG} and~\ref{thm:SEG}
will be the following result.

\begin{theorem}\label{thm:sizeLline} 
For every $k \in \eN$, there a set $S \subseteq \{-,+\}^m$
with $m\leq 12k+37$ and  $|S| \leq 33k+103$ such that 
\begin{enumerate}
\item\label{itm:sizeL1} There exists a line arrangement $\Lcal$ with $S \subseteq \Dcal(\Lcal)$;
\item\label{itm:sizeL2} For every line arrangement with $S \subseteq \Dcal(\Lcal)$ we have
$\spanH(\Lcal) \geq 2^{2^k}$.
\end{enumerate}
\end{theorem}

\begin{proof}
Let $\Pcal = (p_1,\dots, p_n)$ be a constructible euclidean point configuration 
such that $\cross(p_1,p_n,p_2,p_5) = 2^{2^{k+1}}$ and $n \leq 3k+9$. 
Such a point configuration exists by Lemma~\ref{lem:GPS}.
We can assume without loss of generality that all the points of $\Pcal$ are distinct
(if a point occurs more than once then we can drop all but occurences after the first from the
sequence $(p_1,\dots, p_n)$ and we will still have a constructible point configuration).
We shall first construct an auxiliary oriented line arrangement $\Lcal$ on $4n+1$ lines, and 
an auxiliary point configuration $\Qcal$ of $11n+4$ points.

Our construction is inspired by the proof of Lemma 4 in~\cite{Shor91}. 
For each $i \geq 5$ let us fix a 4-tuple $f(i) = (j_1,j_2;j_3,j_4)$ 
such that $\{p_i\} = \ell(p_{j_1},p_{j_2})\cap\ell(p_{j_3},p_{j_4})$
and $i > j_1,j_2,j_3,j_4$.
(In principle there can be many different 4-tuples that define the same point 
$p_i$, but it is useful to fix a definite choice for the construction.)
We will also need $0 < \eps_1 < \eps_2 < \dots < \eps_n$, chosen sufficiently
small for the construction that we are about to follow to work.
(How small exactly depends on the point configuration $\Pcal$.)

For every point of $p_i \in \Pcal$ there will be four oriented lines 
$\ell_{4i-3},\ell_{4i-2},\ell_{4i-1},\ell_{4i}$ such that 
the oriented line arrangement $\Lcal_i := (\ell_{4i-3},\ell_{4i-2},\ell_{4i-1},\ell_{4i})$ 
is isomorphic to the one shown in figure~\ref{fig:qinit}, and

\begin{equation}\label{eq:Cidef}
p_i \in E_i := \ell_{4i-3}^-\cap\ell_{4i-2}^-\cap\ell_{4i-1}^-\cap\ell_{4i}^-. 
\end{equation}

\begin{figure}[!ht]
\begin{center}
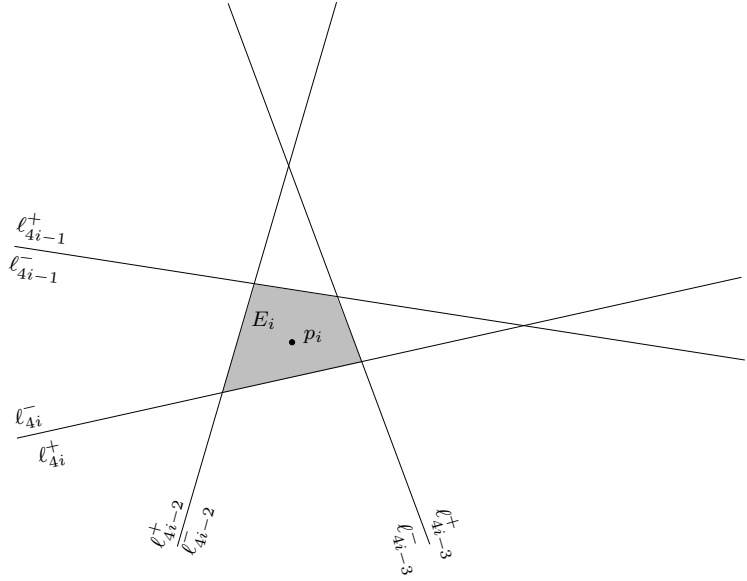
\end{center}
\caption{The oriented line arrangement $\Lcal_i$.\label{fig:qinit}}
\end{figure}

\noindent
So in particular $E_i$ is a quadrilateral
with opposite sides on $\ell_{4i-3}$ and $\ell_{4i-2}$; and 
the other pair of opposite sides on $\ell_{4i-1}$ and $\ell_{4i}$.
%%%, with $\ell_{4i-1}$ and $\ell_{4i}$ determining opposite sides.
There will also be eleven points $q_{11i-10},\dots, q_{11i}$ of $\Qcal$ associated with $p_i$, one
in each cell of $\Lcal_i$.
For notational convenience, let us write $Q_i := (q_{11i-10},\dots,q_{11i})$ and let 
$\Qcal_i^E$ denote the set of those points of $\Qcal_i$ denote that 
lie in cells of $\Lcal_i$ sharing at least one corner point with $E_i$ (so $\Qcal_i^E$ has nine
elements).
In the construction we shall make sure that the following
demands are met:  

\begin{itemize}
\item[\Lcon{1}] $\Qcal_i^E \subseteq B(p_i, \eps_i)$ for all $i=1,\dots,n$;
%%\item[\Lcon{2}] $p_i \in E_i$ and $p_j \not\in E_i$ for all $1\leq i\neq j \leq n$;
%%\item[\Lcon{3}] $p_i \not\in \ell_j$ for all $1\leq i \leq n$ and $1\leq j\leq 4n$;
\item[\Lcon{2}] For each $1\leq i \neq j \leq n$ there is some $4i-3\leq k \leq 4i$
such that $B(p_j,\eps_j) \subseteq \ell_k^+$. 
\end{itemize}

\noindent
Let us now begin the construction of $\Lcal$ and $\Qcal$ in earnest.
To avoid treating definitions of $\Qcal_i, \Lcal_i$ with $i=1,\dots,4$ as special cases, it is 
convenient to define points $p_{0},p_{-1},\dots,p_{-15}$ (no three collinear)
such that $\{p_i\} = \ell(p_{-4i+1},p_{-4i+2}) \cap \ell(p_{-4i+3}, p_{-4i+4})$
for each $i=1,\dots,4$.
We then set $f(i) = (-4i+1,-4i+2;-4i+3,-4i+4)$ for $i=1,\dots,4$; and 
$\Qcal_{j} := (p_{j},\dots,p_{j})$ and $E_{j} = \{p_{j}\}$ for $j\leq 0$.

Suppose that, for some $i\geq 1$, we have already defined $\Qcal_j$
 and $\Lcal_j$ for all $j < i$ and that, thus far, the demands \Lcon{1}-\Lcon{2} are met.
Let $f(i)=(j_1,j_2;j_3,j_4)$. 
We shall place $\ell_{4i-3},\ell_{4i-2}$ 
both at a very small angle to $\ell(p_{j_1},p_{j_2})$ 
in such a way that:
\begin{itemize}
\item[a)] if $p_i$ lies on the segment $[p_{j_1}, p_{j_2}]$ then 
$\ell_{4i-3},\ell_{4i-2}$ intersect in a point $\in \ell(p_{j_1},p_{j_2}) \setminus
[p_{j_1}, p_{j_2}]$ and
$p_i, \Qcal_{j_1}^E,\Qcal_{j_2}^E$ will lie in the same cell of the
line arrangement $(\ell_{4i-3},\ell_{4i-2})$;
\item[b)] if $p_i$ lies outside the segment $[p_{j_1},p_{j_2}]$ then 
$\ell_{4i-3},\ell_{4i-2}$ intersect in a point $\in [p_{j_1}, p_{j_2}]$ and
$\Qcal_{j_1}^E,\Qcal_{j_2}^E$ will lie in opposite (i.e.~not sharing a segment) cells of the
line arrangement $(\ell_{4i-3},\ell_{4i-2})$ and $p_i$ will lie either in the same
cell as $\Qcal_{j_1}^E$ or in the same cell as $\Qcal_{j_2}^E$.
\end{itemize}

\noindent
(See figure~\ref{fig:qaddline}.)

\begin{figure}[!ht]
\begin{center}
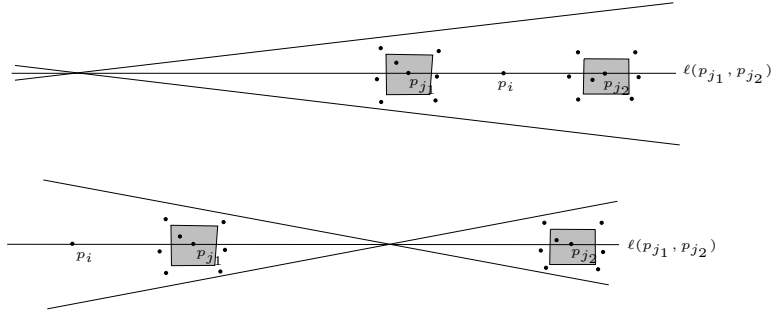
\end{center}
\caption{Placing $\ell_{4i-3},\ell_{4i-2}$ depending on
whether $p_i$ lies on the segment $[p_{j_1},p_{j_2}]$ or not.
The grey quadrangles represent $E_{j_1}$ and $E_{j_2}$.\label{fig:qaddline}}
\end{figure}

\noindent
We place $\ell_{4i-1},\ell_{4i}$ in relation to $\Qcal_{j_3}^E,\Qcal_{j_4}^E$ 
analogously according to whether $p_i \in [p_{j_3},p_{j_4}]$ or not
(again both at a small angle to $\ell(p_{j_3},p_{j_4})$).

Observe that, provided $\eps_1,\dots,\eps_{i-1}$ were small enough, we can
place the lines $\ell_{4i-3},\dots,\ell_{4i}$ such that a) and b) above hold and 
in addition the angles between $\ell(p_{j_1},p_{j_2})$
and $\ell_{4i-3},\ell_{4i-2}$ and the angles between $\ell(p_{j_3},p_{j_4})$ and $\ell_{4i-1}, \ell_{4i}$ 
are small enough to make sure that:

\begin{enumerate}
\item[1)] $\Lcal_i = (\ell_{4i-3},\ell_{4i-2},\ell_{4i-1},\ell_{4i})$ is isomorphic to the 
oriented line arrangement in figure~\ref{fig:qinit};
\item[2)] $p_i$ lies in the quadrangular cell $E_i$ of $\Lcal_i$;
\item[3)] $\clo(E_i) \subseteq B(p_i,\eps_i)$.
\end{enumerate}

\noindent
(Here $\clo(.)$ denotes topological closure.) 
We can then orient the lines $\ell_{4i-3},\ell_{4i-2},\ell_{4i-1},\ell_{4i}$ 
in such a way that $E_i = \ell_{4i-3}^-\cap\ell_{4i-2}^-\cap\ell_{4i-1}^-\cap\ell_{4i}^-$.)
%%%is the quadrangular cell of $\Lcal_i$.)
We now place $\Qcal_i$ in such a way that $\Qcal_i^E \subseteq B(p_i,\eps_i)$ 
(recall that $\Qcal$ has one point in each cell of $\Lcal_i$ and $\Qcal_i^E$ consists of those points of $\Qcal_i$ in cells
sharing at least one corner with $E_i$). Because of 3) this is possible.  
Thus, \Lcon{1} holds up to $i$.
To see that we can also satisfy \Lcon{2}, notice that for each $j\neq i$,
either $p_j \not\in \ell(p_{j_1},p_{j_2})$ or $p_j\not\in \ell(p_{j_3},p_{j_4})$, because
otherwise we would have $p_j=p_i$. 
Without loss of generality $p_j \not\in \ell(p_{j_1},p_{j_2})$.
We can then also assume that $\eps_j$ was chosen such that $B(p_j,\eps_j)$ misses $\ell(p_{j_1},p_{j_2})$, and 
hence if we place $\ell_{4i-3},\ell_{4i-2}$ close enough to $\ell(p_{j_1},p_{j_2})$ then
either $B(p_j,\eps_j) \subseteq \ell_{4i-3}^+$ or $B(p_j,\eps_j) \subseteq \ell_{4i-2}^+$.

Let $A_i$ denote the cell of $\Lcal_i$ that contains $\Qcal_{j_1}^E$;
let $B_i$ denote the cell of $\Lcal_i$ that contains $\Qcal_{j_2}^E$; let $C_i$ denote the cell of $\Lcal_i$ 
that contains $\Qcal_{j_3}^E$; and let $D_i$ denote the cell of $\Lcal_i$ that contains $\Qcal_{j_4}^E$.
Observe that the situation must be one of the situations 
as in figure~\ref{fig:ABCD}, up to swapping of the labels $A$ and $B$ and/or swapping of the labels 
$C$ and $D$.

\begin{figure}[!ht]
\begin{center}
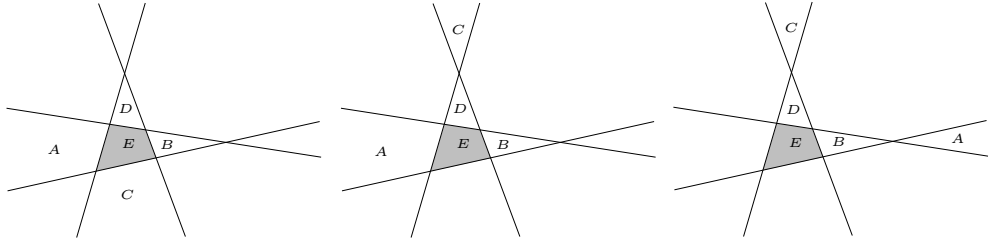
\end{center}
\caption{The different positions of
$Q_{j_1}^E, \dots, Q_{j_4}^E$ in the cells of $\Lcal_i$.
The left figure corresponds to case a) twice, the middle figure to case a) once and case b) once, and
the right figure to case b) twice.\label{fig:ABCD}}
\end{figure}

\noindent
We now set $\ell_{4n+1} := \ell(p_1,p_2)$ (oriented in an arbitrary way) 
and we pick 

\begin{equation}\label{eq:11n+1def}
\begin{array}{l}
q_{11n+1} \in E_1 \cap \ell_{4n+1}^-, \quad q_{11n+2}\in E_1\cap \ell_{4n+1}^+, \\
q_{11n+3} \in E_2 \cap \ell_{4n+1}^-, \quad q_{11n+4}\in E_2\cap \ell_{4n+1}^+.
\end{array}
\end{equation}

\noindent
To finalize the construction, let us set:
\[ 
S := \{ \sigma(q_i;\Lcal) : 1\leq i \leq 11n+4 \}
\]

\noindent
(So trivially $S \subseteq \Dcal(\Lcal)$.)

% To see that \Lcon{2} is satisfied, let $A$ denote the cell that $\Qcal_{j_1}$ lies in,
% $B$ the cell that $\Qcal_{j_2}$ lies in, $C$ the cell that $\Qcal_{j_3}$ lies in, 
% and $D$ the cell that $\Qcal_{j_4}$ lies in.
% The situation must be one of the situations 
% as in figure~\ref{fig:ABCD}, up to swapping of $A$ and $B$ and $C$ and $D$.
% For every $a\in A, b\in B, c\in C, d\in D$, the lines
% $\ell(a,b)$ and $\ell(c,d)$ intersect in a point in $E = E_i$.
% It follows from~\eqref{eq:Cconv} that
% $E_{j_1} \subseteq A, E_{j_2}\subseteq B, E_{j_3}\subseteq C, 
% E_{j_4}\subseteq D$, so that \Lcon{2} indeed holds.

Now let $\Lcaltil = (\elltil_1,\dots,\elltil_{4n+1})$ be
an oriented line arrangement with $S \subseteq \Dcal(\Lcaltil)$.
Let us fix points $\Qcaltil = (\qtil_1, \dots, \qtil_{11n+4})$ 
with $\sigma(\qtil_i;\Lcaltil) = \sigma(q_i;\Lcal)$ for all 
$i=1,\dots,11n+4$ and let
$\Qcaltil_i, \Qcaltil_i^E, \Lcaltil_i$ be defined in the obvious way. 
Observe that for each $i=1,\dots,n$

\[ 
\{ \sigma( q ; \Lcal_i ) : q \in \Qcal_i \}
=
\{ \sigma( q ; \Lcaltil_i ) : q \in \Qcaltil_i \}, 
\]

\noindent
so that, using Lemma~\ref{lem:simpleiso}, $\Lcal_i$ and $\Lcaltil_i$ are isomorphic.
In particular $\Lcaltil_i$ is again isomorphic to the oriented line arrangement
shown in figure~\ref{fig:qinit}.
Let us thus define $\Atil_i,\Btil_i,\Ctil_i,\Dtil_i,\Etil_i$ 
as the cells of $\Lcaltil_i$ corresponding to $A_i,B_i,C_i,D_i,E_i$.
Observe that (see figure~\ref{fig:ABCD}):

\begin{equation}\label{eq:abcd}
\begin{array}{l}
\text{For all } a\in\Atil_i,b\in\Btil_i,c\in\Ctil_i,d\in\Dtil_i
\text{ the lines } \ell(a,b), \ell(c,d) \\
\text{intersect in  a point } e\in\Etil_i.
\end{array}
\end{equation}

\noindent
This observation shall play a key role below.

It follows from \Lcon{1}-\Lcon{2} that for every $i\neq j$ there
is a $4i-3\leq k\leq 4i$ such that $\Qcal_j^E \subseteq \ell_k^+$.
We must then also have $\Qcaltil_j^E \subseteq \elltil_k^+$.
By convexity, this also gives $\conv(\Qcaltil_j^E)\subseteq \elltil_k^+$.
Because $\Qcaltil_j^E$ contains a point in each cell of $\Lcaltil_j$ sharing at least one
corner with $\Etil_j$, we have that 

\begin{equation}\label{eq:Etili}
\clo(\Etil_i) \subseteq \conv(\Qcaltil_i^E) \text{ for all } i=1,\dots,n, 
\end{equation}

\noindent
which implies

\begin{equation}\label{eq:Etildisj}
\clo(\Etil_i) \cap \clo(\Etil_j) = \emptyset \text{ for all } 1 \leq i\neq j \leq n.
\end{equation}

\noindent
%%%(Here $\clo(.)$ denotes topological closure.)
Also observe that, from~\eqref{eq:11n+1def} it follows that 
$\elltil_{4n+1}$ intersects both $\Etil_1$ and $\Etil_2$.

We will now construct a point set
$\Pcaltil =(\ptil_1,\dots,\ptil_n)$ that will turn out to be projectively equivalent
to $\Pcal$.
We first pick $\ptil_1\in\elltil_{4n+1}\cap\Etil_1$ and $\ptil_2\in\elltil_{4n+1}\cap\Etil_2$, and then
we pick $\ptil_3\in\Etil_3,\ptil_4\in\Etil_4$ in such a way that $\ptil_1,\dots,\ptil_4$ are 
in general position (this can clearly be done because $\Etil_3,\Etil_4$ are nonempty and open). 
%%%such that no three are collinear (since the $\Etil_i$s are non-empty and open this can clearly be done).
Once $\ptil_1,\dots,\ptil_{i-1}$ have been constructed for some $i \geq 5$, we
set

\[ 
\{\ptil_i\} := \ell(\ptil_{j_1},\ptil_{j_2})\cap\ell(\ptil_{j_3},\ptil_{j_4}), 
\]

\noindent
where $f(i) = (j_1,j_2;j_3,j_4)$.
Since $\Qcal_{j_1}^E\subseteq A_i, \Qcal_{j_2}^E\subseteq B_i, \Qcal_{j_3}^E\subseteq C_i, 
\Qcal_{j_4}^E\subseteq D_i$ it follows from~\eqref{eq:Etili} that
$\Etil_{j_1} \subseteq \Atil_i, \Etil_{j_2}\subset\Btil_i, \Etil_{j_3}\subseteq\Ctil_i,
\Etil_{j_4}\subseteq\Dtil_i$. 
Applying the observation~\eqref{eq:abcd} gives that $\ptil_i\in\Etil_i$.

By Lemma~\ref{lem:4ptsprojeq} there is a projective transformation $T$ that 
maps $p_i$ to $\ptil_i$ for $i=1,\dots,4$.
We now claim that in fact we must have $T(p_i) = \ptil_i$ for all $i=1,\dots,n$.
To see this suppose that, for some $i\geq 5$, we have $T(p_j)=\ptil_j$ for all $j<i$.
Let us again write $f(i) = (j_1,j_2;j_3,j_4)$. 
Since projective transformations map lines to lines, 
we have that $T[\ell(p_{j_1},p_{j_2})] = \ell(\ptil_{j_1},\ptil_{j_2})$
and $T[\ell(p_{j_3},p_{j_4})] = \ell(\ptil_{j_3},\ptil_{j_4})$.
This implies that indeed $T(p_i) = \ptil_i$.
The claim follows.

%%%Also note that  $\ptil_1,\ptil_2,\ptil_m,\ptil_n$ are collinear 
%%%since they are on the line $T[\ell(p_1,p_2)] = \ell(\ptil_1,\ptil_2) = \elltil_{4n+1}$.
Using Lemma~\ref{lem:crossproj} we find

\[
\cross( \ptil_5,\ptil_n,\ptil_2,\ptil_1)
= \cross( p_5,p_n,p_2,p_1) = 2^{2^{k+1}}.
\]

\noindent
Thus, by Lemma~\ref{lem:crosseucl}, we have either
$\norm{\ptil_5-\ptil_2} / \norm{\ptil_5-\ptil_1} \geq \sqrt{2^{2^k+1}} = 2^{2^k}$ or
$\norm{\ptil_n-\ptil_2} / \norm{\ptil_n-\ptil_1} \geq 2^{2^k}$.
Without loss of generality $\norm{\ptil_n-\ptil_2} / \norm{\ptil_n-\ptil_1} \geq 2^{2^k}$.
Observe that $\ptil_1,\ptil_2, \ptil_n \in \elltil_{4n+1} = \ell(\ptil_1,\ptil_2)$ 
as $p_1,p_2,p_n \in \ell_{4n+1} = \ell(p_1,p_2)$.
Let the segments $I_1,I_2,I_n$ be defined by
$I_j := \elltil_{4n+1}\cap\Etil_j$ ($j=1,2,n$).
%%%Then the $I_j$s are disjoint by~\eqref{eq:Etildisj}. 
%%%Since $\ptil_2,\ptil_n \in\elltil_{4n+1}$, 
The distance $\norm{\ptil_2-\ptil_n}$ is at most 
the furthest distance between an endpoint of $I_2$ and an endpoint of $I_n$.
Similarly, $\norm{\ptil_1-\ptil_n}$ is at least the shortest distance between an endpoint 
of $I_1$ and an endpoint of $I_n$, and this distance is positive by~\eqref{eq:Etildisj}.
Since the endpoints of the $I_j$s are intersection points of $\elltil_{4n+1}$ with 
some other lines of $\Lcaltil$, we see that $\spanH(\Lcaltil) \geq 2^{2^k}$, as required.
\end{proof}

Another ingredient we need for the proofs of the lower bounds is the following 
lemma relating the span of oriented line arrangements to the numbers used
to express the oriented line arrangement as linear inequalities.

\begin{lemma}\label{lem:gridspan}
Let $\Lcal = (\ell_1,\dots,\ell_n)$ be an oriented line arrangement.
Suppose that, for some $k \in \eN$ there are nonzero $w_1,\dots,w_{n} \in \{-k,\dots,k\}^2$ 
and $c_1,\dots,c_{n} \in \{-k,\dots,k\}$ such that we can express the lines as:

\[ 
\ell_i = \{ z : w_i^Tz = c_i \}.
\]

\noindent
Then $\spanH(\Lcal) \leq 2^{9/2} \cdot k^6$.
\end{lemma}

\begin{proof}
Any point $p\in\Ical(\Lcal)$ is the solution to a $2\times 2$ linear system
$Az=b$.
More precisely, if $\{p\} = \ell_{i} \cap \ell_{j}$ then 

\[
A = \left(\begin{array}{cc}
(w_i)_x & (w_i)_y \\
(w_j)_x & (w_j)_y
\end{array}\right), \quad
b = \left( \begin{array}{c} c_i \\ c_j \end{array} \right),
\]

\noindent
(Observe that, since $p$ must be the unique solution, $A$ is non-singular.)
From the familiar formula

\[ \left(\begin{array}{cc} 
a_{11} & a_{12} \\
a_{21} & a_{22} 
\end{array}
\right)^{-1} 
=
\left(\begin{array}{cc} 
\frac{a_{22}}{a_{11}a_{22}-a_{12}a_{21}} & \frac{-a_{12}}{a_{11}a_{22}-a_{12}a_{21}} \\
\frac{-a_{21}}{a_{11}a_{22}-a_{12}a_{21}} & \frac{a_{11}}{a_{11}a_{22}-a_{12}a_{21}} 
\end{array}
\right),
\]

\noindent
we see that $p = A^{-1}b$ has coordinates
$|p_x|,|p_y| \leq 2k^2$; and 
both coordinates are ratios $\frac{s}{t}$ of two integers with denominator 
$1 \leq t \leq 2k^2$. 
Hence we have 

\[
\max_{p,q\in\Ical(\Lcal)} \norm{p-q} \leq 4k^2\sqrt{2}.
\]

\noindent
Similarly, because if $p\neq q$ then either
$p_x\neq q_x$ or $p_y\neq q_y$, we have (also recall $\frac{s_1}{t_1}-\frac{s_2}{t_2} = \frac{s_1t_2-s_2t_1}{t_1t_2}$):

\[ 
\min_{p,q\in\Ical(\Lcal),\atop p\neq q} \norm{p-q} \geq \frac{1}{4k^4}.
\]

\noindent
The lemma follows.
\end{proof}

\section{The lower bound for unit disk graphs\label{sec:UDGlb}}

\noindent
For convenient reference later on, we have separated out the following observation as a lemma.

\begin{lemma}\label{lem:embeddisksR}
Let $\Lcal = (\ell_1,\dots,\ell_n)$ be an oriented line arrangement and 
let $\Pcal \subseteq \eR^2$ %%% \setminus \bigcup_{i=1}^n \ell_i$ 
be a finite set of points.
%%%%with $\sigma(p_i) \in \{-,+\}^n$ for all $i$.
Then there exist equal radius disks
$D_1^-,D_1^+,\dots,D_n^-,D_n^+$ such that 
$\Pcal \cap \ell_i^- \subseteq D_i^- \subseteq  \ell_i^-$ and
$\Pcal \cap \ell_i^+ \subseteq D_i^+ \subseteq \ell_i^+$ for
all $i=1,\dots,n$.
\end{lemma}

\begin{proof}
We first claim that for each $1\leq i \leq n$ and each $s \in \{-,+\}$ there is 
an $R_0(i,s)$ such that for all $R > R_0(i,s)$, there exists a
point $q_{i}^{s}(R)$ such that 
$\Pcal \cap \ell_i^{s} \subseteq B(q_i^{s}(R);R) \subseteq  \ell_i^{s}$.

To see this, fix an $1\leq i \leq n$ and an $s\in\{-,+\}$.
Let us write $\Pcal \cap \ell_i^{s} =: \{a_1,\dots, a_m\}$.
Observe that by applying a suitable isometry if needed we can assume without loss of 
generality that $\ell_i^{s} = \{ (x,y)^T : y > 0 \}$ and, thus $a_j = (x_j,y_j)^T$ with 
$y_j > 0$ for all $1 \leq j \leq m$.
Let us set $q(R) := (0,R)^T$. 
We have $\norm{q(R) - a_j}^2 = x_j^2 + (R - y_j)^2 = R^2 - 2 y_j R + (x_j^2+y_j^2)$, and therefore
$\norm{q(R)-a_j} < R$ for $R > \norm{a_j}^2/2y_j$. %%(recall $y_j > 0$).
In other words, for $R > R_0 := \max_j \frac{\norm{a_j}^2}{2y_j}$ we have $a_j \in B(q(R);R) \subseteq \ell_+$
for all $j$, proving the claim.

% Let us write $\Pcal = \{ p_1,\dots, p_m\}$.
% Pick an $1\leq i\leq n$ and let us set
% $\Pcal_i^- := \Pcal \cap \ell_i^-, \Pcal_i^+ := \Pcal\cap\ell_i^+$.
% Thus, for each $1\leq i \leq n$ and each $\sigma \in \{-,+\}$ there is 
% a $R_0(i,\sigma)$ such that for all $R \geq R_0(i,\sigma)$, there exists a
% point $q_{i}^{\sigma}(R)$ such that $\Pcal \cap \ell_i^{\sigma} \subseteq B(q_i^{\sigma}(R);R) \subseteq  \ell_i^{\sigma}$.
If we pick $R > \max_{i,s} R_0(i,s)$ %%$R > \max_{i=1,\dots,n,\atop\sigma \in \{-,+\}} R_0(i,\sigma)$
and we set $D_i^s = B( q_i^{s}(R); R)$ then the lemma follows.
\end{proof}

\noindent
The following proposition allows us to encode a combinatorial description of an oriented line arrangement into
a unit disk graph.

\begin{lemma}\label{lem:UDGembed}
Let $\Lcal = (\ell_1,\dots,\ell_n)$ be an oriented line arrangement and $\Scal \subseteq \Dcal(\Lcal)\cap\{-,+\}^n $.
There exists a unit disk graph $G$ on 
vertex set

\[ V(G) = \{v_i^-,v_i^+ : i=1,\dots,n\} \cup 
\{ u_j : j = 1,\dots,|\Scal| \}, \] 

\noindent
such that 
%$2n+|\Scal|$ vertices, such that 
in any $\UDG$-realization $(B(\ptil(v),r) : v\in V(G) )$ of $G$, the 
oriented line arrangement $\Lcaltil = (\elltil_1,\dots, \elltil_n)$ defined by 

\[ 
\elltil_i^- := \{ z : \norm{ z - \ptil(v_i^-)} < \norm{z-\ptil(v_i^+)} \}, %\\
\]

\noindent
for $i=1,\dots,n$, has $\Scal \subseteq \Dcal(\Lcaltil)$. 
\end{lemma}

\begin{proof}
%%%Let $\Lcal$ be a realization of $\Dcal$, and let
Let $\Pcal = \{ p_1,\dots,p_{|\Scal|}\}$ be a set of points
such that $\Scal = \{ \sigma(p) : p \in \Pcal \}$.

Let $D_1^-,D_1^+, \dots, D_n^-,D_n^+$ be as provided by Lemma~\ref{lem:embeddisksR}, let
$R$ denote their common radius and let $p(v_i^s)$ denote the center of $D_i^s$ for
$1\leq i \leq n$ and $s \in \{-,+\}$.
Define

\begin{equation}\label{eq:GUDdef}
\begin{array}{l}
D(v_i^s) := B( p(v_i^s); R/2 ) \quad \text{ for } s \in \{-,+\}, i=1,\dots,2n, \\
D(u_j) := B(p_j;R/2) \quad \text{ for } j=1,\dots,|\Scal|;
\end{array} 
\end{equation}

\noindent
and let $G$ be the corresponding intersection graph.
Observe that there is an edge between 
$v_i^s$ and $u_j$ if and only if $p_j \in B(p(v_i^s); R ) = D_i^s$; and this
happens if and only if $p_j \in \ell_i^s$ by choice of the $D_i^s$s.
%%Similarly there is an edge between $2i$ and $2n+j$ if and only if
%%$p_j \in \ell_i^+$. 

Now let $(B(\ptil(v),r): v \in V(G))$ be an arbitrary realization of $G$
as the intersection graph of disks of equal radius $r$, and let $\Lcaltil$ be as in the statement of the 
lemma.
Pick a $1 \leq j \leq |\Scal|$ and a $1\leq i \leq n$, and suppose that $p_j \in \ell_i^-$.
Then, by definition~\eqref{eq:GUDdef} of $G$, we must have that $B(\ptil(v_i^-), r) \cap B(\ptil(u_j),r) \neq \emptyset$ and
$B(\ptil(v_i^+), r) \cap B(\ptil(u_j),r) = \emptyset$.
This gives $\norm{\ptil(u_j)-\ptil(v_i^-)} < \norm{\ptil(u_j)-\ptil(v_i^+)}$.
Or, in other words, $\ptil(u_j) \in \elltil_i^-$.
Similarly we have $\ptil(u_j) \in \elltil_i^+$ if $p_j \in \ell_i^+$.

Since this holds for all $1\leq i \leq n$, we see that

\[ 
\sigma( \ptil(u_j);\Lcaltil) = \sigma( p_j;\Lcal), 
\]

\noindent
for all $j=1,\dots,|\Scal|$. Hence $\Scal \subseteq \Dcal(\Lcaltil)$ as required.
\end{proof}

\noindent
We are now in a position to prove:

\begin{lemma}\label{lem:UDGlb} $\fUDG(n) = 2^{2^{\Omega(n)}}$.
\end{lemma}

\begin{proof}
It suffices to show that for every $k\in\eN$ there exists a
unit disk graph $G$ on $O(k)$ vertices with $\fUDG(G) = 2^{2^{\Omega(k)}}$.
Let us thus pick an arbitrary $k \in \eN$,
let $\Lcal, \Scal$ be as provided by Theorem~\ref{thm:sizeLline} and let $G$ be as provided by 
Lemma~\ref{lem:UDGembed}.
Then $|V(G)| = 2|\Lcal|+|\Scal| = O(k)$. 
Let $(B(p(v);r) : v \in V(G))$ be an arbitrary $\UDG$-realization of $G$
with $p(v) \in \{-m,\dots,m\}^2$ for all $v$ and $r \in \{1,\dots,m\}$ for some $m\in\eN$.

Observe that the line $\elltil_i$ from Lemma~\ref{lem:UDGembed} satisfies

\[ \begin{array}{rcl}
\elltil_i 
& = & 
\{ z : \norm{ z - p(v_i^-)} = \norm{z-p(v_i^+)} \} \\
%%& = & 
%%\{ z : (z_{2i}-z_{2i-1})^T z < (z_{2i}-z_{2i-1})^T (\frac12z_{2i}+\frac12z_{2i-1}) \} \\
& = & 
\{ z : (p(v_i^+)-p(v_i^-))^T z = (p(v_i^+)-p(v_i^-))^T {\Big(}\frac{p(v_i^+)+p(v_i^-)}{2}{\Big)} \} \\
& = & 
\{ z : 2 (p(v_i^+)-p(v_i^-))^T z = (p(v_i^+)-p(v_i^-))^T (p(v_i^+)+p(v_i^-)) \} \\
& =: & 
\{ z : w_i z = c_i \}. 
\end{array} \]

\noindent
Observe that the $w_i$s have integer coordinates and the $c_i$s are integers, whose
absolute values are all upper bounded by $8 m^2$. 
We can thus apply Lemma~\ref{lem:gridspan} to get that 

\[ 
2^{9/2} (8m^2)^6
= 2^{45/2} m^{12} \geq \spanH(\Lcaltil) \geq 2^{2^k}.
\]

\noindent
Hence 

\[ 
m 
\geq 
\sqrt[12]{2^{2^k} / 2^{45/2}} 
= 
2^{2^k / 12 - 45/24 } 
= 
2^{2^{k - \log(12) - o(1)}}
=
2^{2^{\Omega(k)}}, 
\]

\noindent
which proves the lemma.
\end{proof}

\section{The lower bound for disk graphs\label{sec:DGlb}}

To prove the lower bound in Theorem~\ref{thm:DG} we develop a construction
for ``embedding" a line arrangement into a disk graph that is analogous to that 
for unit disk graphs in Lemma~\ref{lem:UDGembed}; and then 
we will again apply Lemma~\ref{lem:gridspan}.
In the proof of Lemma~\ref{lem:UDGembed} we used two vertices for every oriented line of $\Lcal$ and one vertex
for each sign vector of $\Scal$.
For disk graphs we can still use two vertices for each line, but rather than a single vertex 
we will need to place an induced copy of a special disk graph $H$ with a certain desirable property 
for each sign vector.
The construction of $H$ is pretty involved and takes up most of this section.
Rather than giving a list of vertices and edges, we will give a (geometric) procedure for constructing a 
realization of $H$ as a disk graph.
But before we can begin the construction of $H$, we will need to do some preliminary work.

By a result of Koebe~\cite{Koebe36} every planar graph is an 
intersection graph of touching disks (i.e.~closed disks with disjoint interiors).
This also gives that every planar graph is the intersection graph of open disks.

\begin{theorem}\label{thm:PlanarIsDG}
Every planar graph is a disk graph. \noproof
\end{theorem}

\noindent
Recall that a planar {\em embedding} of a planar graph
assigns each vertex $v \in V(G)$ to a point $p(v)\in\eR^2$ in the plane,
and each edge $uv\in E(G)$ to a simple closed curve $\gamma(uv)$ with endpoints
$p(u),p(v)$ such that for any distinct $e,f\in E(G)$, the curves $\gamma(e), \gamma(f)$   
do not intersect except possibly in a common endpoint.
%%We will say that a vertex $v \in V(G)$ lies inside the cycle $C$ 
%%in an embedding of $G$ if $p(v)$ is contained
%%in the bounded component of $\eR^2 \setminus \bigcup_{e \in E(C)} c(e)$.

In a {\em F\'ary embedding} the curves $\gamma(e)$ are straight-line segments, i.e.~$\gamma(uv) = [p(u),p(v)]$.
A fundamental result that was proved at least three separate times by Wagner~\cite{Wagner36}, F\'ary~\cite{Fary48} and 
Stein~\cite{Stein51} states that every planar graph has a F\'ary embedding.
The following observation gives a partial converse to Theorem~\ref{thm:PlanarIsDG}.
It is is essentially the same as Theorem 3.4 in Breu's PhD 
thesis~\cite{BreuThesis} and Lemma 4.1 in Malesi\'nska's PhD thesis~\cite{MalesinskaThesis}. 
We give a proof for completeness.

\begin{lemma}\label{lem:trianglefreeplanar}
Let $G$ be a triangle-free disk graph of minimum degree at least two, and let 
$(B(p(v), r(v)) : v\in V(G) )$ be a realization of $G$ as a disk graph.
Then $G$ is planar and the points $p(v)$ define a F\'ary embedding of $G$. 
\end{lemma}

\begin{proof}
Observe first that if $u,v$ are distinct vertices then
$\norm{p(u)-p(v)} > r(u)-r(v)$. For if not then $B(p(v),r(v))\subseteq B(p(u),r(u))$ and
then $G$ would have a triangle since $v$ has degree at least two.

Let $uv, uw \in E(G)$ be two distinct edges that share an endpoint.
We claim that 

\begin{equation}\label{eq:edgeembedcrap}
[p(u),p(v)] \cap [p(u),p(w)] = \{p(u)\}. 
\end{equation}

To see this, let us write $\alpha := \angle p(v)p(u)p(w)$.
If $\alpha \neq 0$ then~\eqref{eq:edgeembedcrap} is easily seen to hold.
Let us thus suppose that $\alpha=0$.
%%Hence $\norm{p(u)-p(v)} > r(u)-r(v)$ and similarly $\norm{p(u)-p(w)} > r(u)-r(w)$.
Then we have either $p(u) \in [p(v),p(w)]$ or $p(v) \in [p(u),p(w)]$ or $p(w)\in [p(u),p(v)]$.
If $p(u) \in [p(v),p(w)]$ then~\eqref{eq:edgeembedcrap} is again easily see to hold.
Let us thus suppose that $p(v) \in [p(u),p(w)]$. 
We have

\[ \begin{array}{rcl}
\norm{p(v)-p(w)} & = & \norm{p(u)-p(w)} - \norm{p(u)-p(v)} \\
& < & r(u)+r(w) - (r(u)-r(v)) \\
& = & r(w)+r(v). 
\end{array} \]

\noindent
But then we must have $vw\in E(G)$, contradicting that $G$ is triangle-free.
Similarly we cannot have $p(w)\in [p(u),p(v)]$.
Thus $\alpha \neq 0$ and hence~\eqref{eq:edgeembedcrap} holds by a previous argument, as claimed.

Now consider two edges $uv, st \in E(G)$ with $u,v,s,t$ distinct.
We claim that

\begin{equation}\label{eq:edgeembedcrap2}
 [p(u),p(v)]\cap[p(s),p(t)] = \emptyset.
\end{equation}

Aiming for a contradiction, let us suppose that the segments $[p(u),p(v)],[p(s),p(t)]$ intersect in some 
point $q$.
Observe that 

\[ \begin{array}{rcl}
\norm{p(u)-p(s)} + \norm{p(v)-p(t)}
& \leq & 
\norm{p(u)-q} + \norm{q-p(s)} \\
& & 
+ \norm{p(v)-q}+\norm{q-p(t)} \\
& = & 
\norm{p(u)-p(v)} + \norm{p(s)-p(t)} \\
& < & 
r(u)+r(v)+r(s)+r(t).
\end{array} \]
 
\noindent
Hence either $\norm{p(u)-p(s)} < r(u)+r(s)$ or
$\norm{p(v)-p(t)} < r(v)+r(t)$.
In other words, either $us \in E(G)$ or $vt \in E(G)$.
Similarly either $ut\in E(G)$ or $vs \in E(G)$.
It is easily checked that in each of the four cases there is
a triangle, contradicting that $G$ is triangle-free.

It follows that~\eqref{eq:edgeembedcrap2} holds, as claimed.
\end{proof}

\begin{lemma}\label{lem:K1k}
For every $\eps>0$ there is a $k=k(\eps)$ such that the following holds.
Let $G = K_{1,k}$ be the star on $k+1$ vertices and let $u \in V(G)$ denote the vertex of degree $k$.
For any $\DG$-realization $(B(p(v),r(v)):v\in V(G))$ of $G$, 
there is a $w\in V(G)$ such that $r(w) < \eps \cdot r(u)$.
\end{lemma}

\begin{proof}
Let $\eps > 0$ be arbitrary and let $k = \lceil (1+\eps)^2/\eps^2\rceil + 1$.
Let $(B(p(v),r(v)): v\in V(K_{1,k}) )$ be an arbitrary realization of 
$K_{1,k}$ as a disk graph, and suppose that 
$r(u)=1$ and $r(v) \geq \eps$ for all $v\neq u$.
We can assume that $r(v) = \eps$ for all $v\neq u$, by replacing
$D(v)$ with a smaller ball $D'(v) := B(p'(v),\eps) \subseteq D(v)$ that still
intersects $D(u)$.
Since $D(v)$ intersects $D(u)$ and has radius $\eps$, we have
$D(v) \subseteq B(p(u),1+\eps)$ for all $v \neq u$.
Since $D(v) \cap D(w) = \emptyset$ for all $v\neq w \in V(G)\setminus \{u\}$, by area considerations 
we must have

\[ 
k \leq \frac{\pi (1+\eps)^2 }{ \pi \eps^2} = (1+\eps)^2 / \eps^2 < k, 
\]

\noindent
a contradiction.
\end{proof}

For a $k \in \eN$ odd, let $O_k$ denote the graph obtained as follows.
We start with a path $u_0,\dots, u_{2k}$ of length $2k$.
Now we add vertices $a,b$ each joined to $u_j$ for all even $j$.
Let $c$ denote $u_k$, the middle vertex of the path, and let us also 
denote the endpoints of the path by $s=u_0, t=u_{2k}$.
See figure~\ref{fig:onion} for a depiction of $O_k$.

\begin{figure}[h!]
\begin{center}
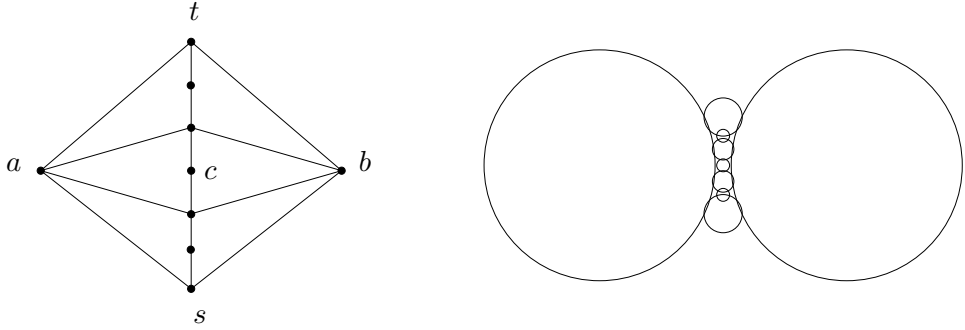
\end{center}
\caption{The graph $O_3$ and a realization of it as a disk graph.\label{fig:onion}}
\end{figure}

\begin{lemma}\label{lem:radiuscsmall}
For every $\eps > 0$ there is a $k=k(\eps)$ such that
if $( B(p(v),r(v)) : v \in V(O_k) )$ is a $\DG$-realization 
of $O_k$, and $p(s), p(a), p(t), p(b)$ lie on the 
outer face of the corresponding F\'ary embedding, then 
$r(c) < \eps\cdot r(a)$.
\end{lemma}

\begin{proof}
Let $k$ be large, and let
 $( B(p(v),r(v)) : v \in V(O_k) )$ be an embedding of $O_k$ as
 a disk graph such that $p(s), p(a), p(t), p(b)$ lie on the 
outer face of the corresponding F\'ary embedding.

By Lemma~\ref{lem:K1k} above, if $k$ was chosen sufficiently large, there is
an even $0 \leq i < k-1$ such that $r(u_i) < \eps\cdot r(a)$ and an even
$k+1 < j \leq 2k$ such that $r(u_j) < \eps\cdot r(a)$.

Observe that, since the outer face of the F\'ary embedding is the quadrilateral 
with corners $p(s), p(a), p(t), p(b)$, we must have
that $p(c)$ lies inside the quadrilateral with vertices
$p(u_i),p(a), p(u_j), p(b)$.

Let us also observe that $\eR^2 \setminus (D(u_i)\cup D(a)\cup D(u_j)\cup D(b))$ consists of two 
connected regions, a bounded and an unbounded one. Let $R$ denote the bounded one, and let
$Q$ denote the (inside of) the quadrangle with corners $p(u_i),p(a), p(j), p(b)$.
Then $R$ is clearly contained $Q$.
%%%%(the segments $[p(u_i), p(a)]$ etc.~are contained in . 
By the previous we have $D(c) \subseteq R$
(as its center $p(c)$ lies in $Q$
and $D(c)$ is disjoint from $D(u_i)\cup D(a)\cup D(u_j)\cup D(b)$).

Observe that $R$ is also contained in the quadrilateral $Q'$
whose corner points are: an intersection point
$q_1$ of $\partial D(u_i)$ and $\partial D(a)$; an intersection point
$q_2$ of $\partial D(a)$ and $\partial D(u_j)$; an intersection point 
$q_3$ of $\partial D(u_j)$ and $\partial D(b)$; and an
intersection point $q_4$ of $\partial D(b)$ and $\partial D(u_i)$
(see figure~\ref{fig:quadrO}).

\begin{figure}[h!]
\begin{center}
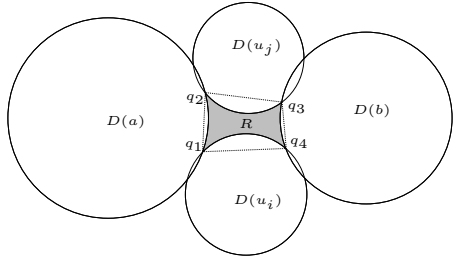
\end{center}
\caption{$R$ is contained in the quadrilateral with 
corners $q_1,q_2,q_3,q_4$.\label{fig:quadrO}}
\end{figure}

\noindent
Clearly $\norm{q_3-q_2} \leq 2 r(u_j)$ and
$\norm{q_4-q_1} \leq 2 r(u_i)$.
Thus, two opposite sides of $Q'$
have length $< 2 \eps \cdot r(a)$.
Since also $D(c) \subseteq Q'$ we then must have $r(c) < \eps \cdot r(a)$, as required. 
\end{proof}

% We will now describe a pretty involved construction of a disk graph $H$ with certain desirable properties
% given by Lemma~\ref{lem:Hlemma} below.
% Rather than giving a list of vertices and edges, we will give a
% (geometric) procedure for constructing a realization of $H$ as a disk graph.

Consider a realization of $O_k$ with $k=k(1/1000)$ as in Lemma~\ref{lem:radiuscsmall} with 
$p(a), p(s), p(t), p(b)$ on the outer face of the corresponding F\'ary embedding
(such a realization is depicted in figure~\ref{fig:onion} for $k=3$).
For notational convenience let us denote $X := O_k$.
%%For convenience let us assume that the clockwise order 
%%of these points is $p(u_0), p(a), p(u_{2k}), p(b)$

We now define a disk graph $Y$ on vertex set%
\[ V(Y) = \{a\} \cup
\{ v_i : v \in V(X)\setminus\{a\}, i=0,\dots,N \}, \] 

\noindent
as follows.
We let $D(a)$ be as in the chosen realization of $X$.
For each $v \in V(X)\setminus\{a\}$ and $i=0,\dots,N$ 
we place a disk $D(v_i) = B(p(v_i),r(v_i))$ where $r(v_{i}) := r(v)$ and
$p(v_{i})$ is obtained by rotating $p(v)$ counterclockwise about $p(a)$ by an angle of $i\cdot\alpha$.
%%;
%%and let us set $D(v_i) = B(p(v_i),r(v_i))$.
Here $\alpha, N$ are chosen so that $C = c_0 \dots c_N$ will constitute an induced
cycle in the resulting intersection graph of disks $Y$. (See figure~\ref{fig:OnionAnim}.)
For notational convenience, let us also write $a_i := a$, and let 
$X_i$ denote the $i$-th rotated copy of $X$
(i.e.~it has vertex set $V(X_i) = \{ v_i : v \in V(X) \}$).
Observe that $D(a)$ does not intersect any $D(c_i)$ and it is
contained in the bounded region of $\eR^2 \setminus \bigcup_i D(c_i)$.
%%surrounded by the cycle made up of the $D(c_i)$s.
%%Let $F$ denote the resulting graph (i.e.~the intersection graph of the
%%the set of disks we have just constructed).
%%For notational convenience we also set $a_i = a$ for $i=0,\dots, N$.
%% and we 
%%denote by $O_k^{(i)}$ denote the $i$th rotated copy of $O_k$;

\begin{figure}[h!]
\begin{center}
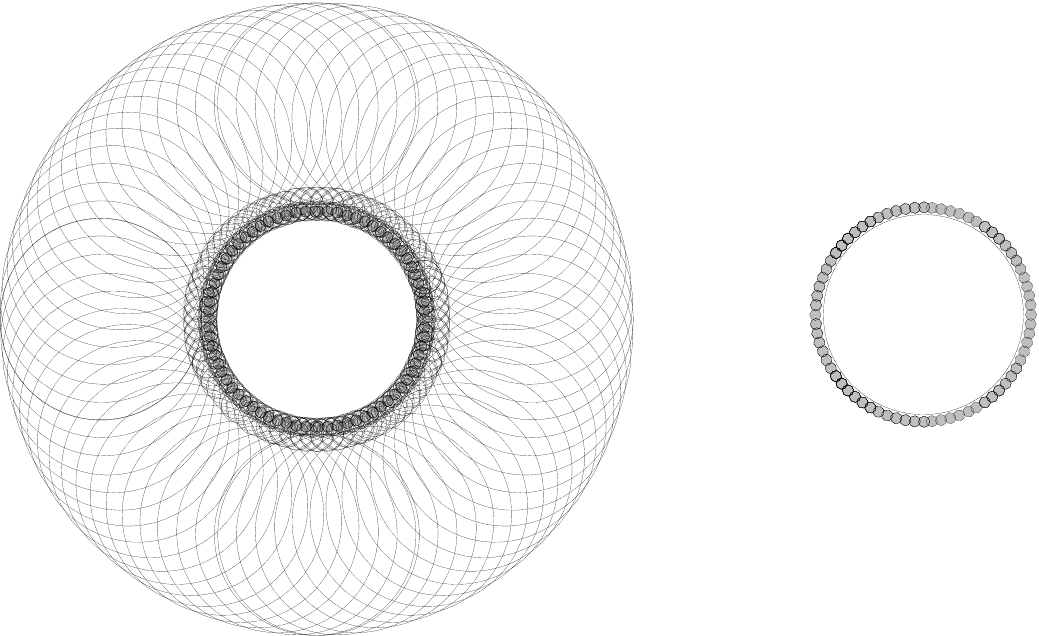
\end{center}
\caption{The graph $Y$ is obtained by rotating copies of a realization of $X$ about $p(a)$ 
in such a way that the $c_i$s form an induced cycle.
On the right only $D(a)$ and the $D(c_i)$s are shown.\label{fig:OnionAnim}}
\end{figure}

We now construct the disk graph $H$ as follows.
We start with a four cycle $F$ %%%(in clockwise order)
and consider a realization of it by equal size disks. 
Inside each of the two regions of 
$\eR^2 \setminus \bigcup_{v\in F} D(v)$ we place a suitably shrunken
copy of the realization of $Y$ we have just constructed.
(See figure~\ref{fig:grid}.)

\begin{figure}[h!]
\begin{center}
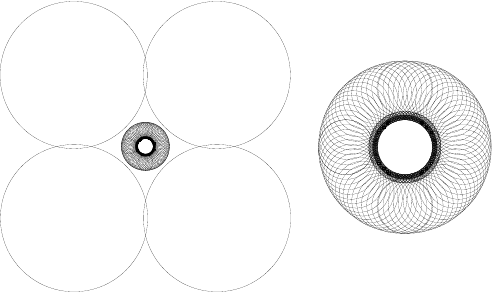
\end{center}
\caption{Placing copies of $Y$ inside each region of
$\eR^2 \setminus \bigcup_{v\in F} D(v)$.\label{fig:grid}}
\end{figure}

Let $Y^{(1)}, Y^{(2)}$ denote these copies of $Y$; for each vertex $v \in V(Y)$ let 
$v^{(j)}$ denote the corresponding vertex in $Y^{(j)}$; and let $C^{(j)}$ and $X_i^{(j)}$
be defined in the obvious way.

We now add four internally vertex disjoint (meaning they do not share vertices other than 
their endpoints) paths $P_1, \dots, P_4$ to our construction %%%, with the new vertices represented by tiny disks, 
such that each of them joins $a^{(1)}$ to $a^{(2)}$ and passes
through a vertex of $C^{(1)}$, a vertex of $F$, and a vertex of $C^{(2)}$; 
and the subgraph $Z$ of $H$ induced by the vertices

\[ 
V(Z) := 
\{a^{(1)},a^{(2)}\} \cup V(F) \cup V(C^{(1)}) \cup V(C^{(2)}) 
\cup \bigcup_{m=1}^4 V(P_m),
\]

\noindent 
is triangle free.
We can do this if we represent the vertices we are adding by small enough disks -- see figure~\ref{fig:addingpathsP}.
Only the stated properties of the paths $P_1,\dots,P_4$ will play a role in the
proof of Lemma~\ref{lem:Hlemma} below; the lengths and any additional edges 
that may have been created inadvertently are irrelevant as long as the induced subgraph 
$Z$ is triangle free and the paths $P_1,\dots,P_4$ have the properties stated.
% that are in addition to the ones described will not play
% any role in the proof of Lemma~\ref{lem:Hlemma} below.

%\includegraphics[width=5cm]{AddPathsP2.mps}

\begin{figure}[h!]
\begin{center}
%%\includegraphics{AddPathsP2.mps}
%%\convertMPtoPDF{AddPathsP2.mp}{1}{1}
%%%\input{AddPathsP2.latex}
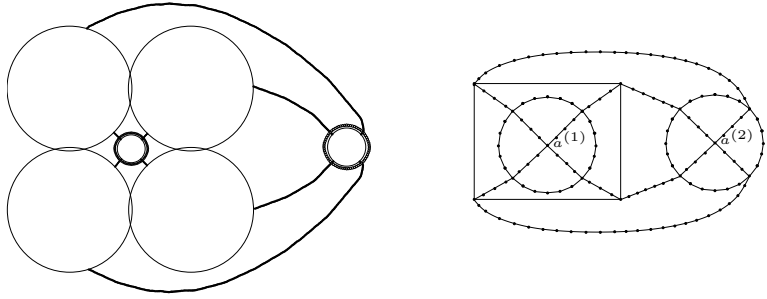
\end{center}
\caption{Adding paths $P_1, \dots, P_4$ between $a^{(1)}$ and $a^{(2)}$.\label{fig:addingpathsP}}
\end{figure}

For $0 \leq i \leq N$ and $j=1,2$, let 
us add vertex disjoint paths 
$Q_1^{(i,j)}, \dots, Q_4^{(i,j)}$ to our construction, each 
running from one of the vertices $a^{(j)}, b_i^{(j)}, s_i^{(j)}, t_i^{(j)}$
to a vertex on $F$, in such a way that 
the subgraph $H^{(i,j)}$ of $H$ induced by the vertices 

\[ 
V(H^{(i,j)}) := 
V(F) \cup V(X_i^{(j)}) \cup \bigcup_{m=1}^4 V(Q_m^{(i,j)}),
\]

\noindent 
is triangle free.
Again this is possible if we choose the radii of the disks making 
up the internal vertices of the paths small enough -- see figure~\ref{fig:addingpathsQ}.

\begin{figure}[h!]
\begin{center}
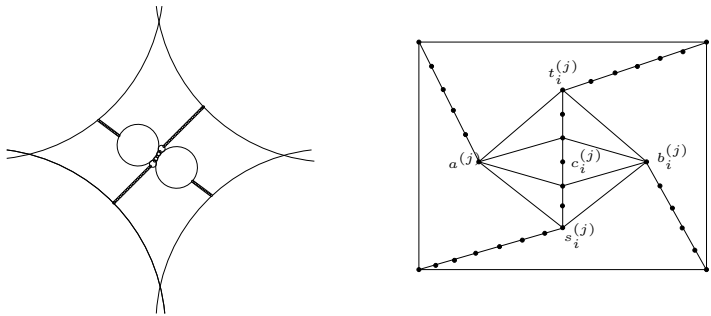
\end{center}
\caption{Adding paths $Q_1^{(i,j)},\dots,Q_4^{(i,j)}$ between $F$ and 
$a^{(j)}, b_i^{(j)}, s_i^{(j)}, t_i^{(j)}$.\label{fig:addingpathsQ}}
\end{figure}

This concludes the construction of $H$.
The following lemma gives the key property of $H$ that will be crucial 
in the proof of the lower bound in Theorem~\ref{thm:DG}.

\begin{lemma}\label{lem:Hlemma}
Let $\Dcal = (D(v):v\in V(H))$ be an arbitrary realization of $H$ as
a disk graph.
Then there exists a point $p = p(\Dcal)$ such that
the following holds for all convex $W \subseteq \eR^2$:
\begin{enumerate}
\item\label{itm:Hlemma.i} 
If $W \cap D(v) \neq \emptyset$ for all $v\in V(H)$ then $p \in W$;
\item\label{itm:Hlemma.ii} 
If $W \cap D(v) = \emptyset$ for all $v \in V(H)$ then $p\not\in W$.
\end{enumerate}
\end{lemma}

\begin{proof}
For $uv \in E(H)$ let us write $\gamma(uv) = [p(u),p(v)]$ and 
for $H' \subseteq H$ a subgraph let us write  $\gamma(H') := \bigcup_{e\in H'} \gamma(e)$.
If $H'$ is an induced cycle of $H$, then $\gamma(H')$ is a simple closed curve
(by Lemma~\ref{lem:trianglefreeplanar}) and hence
$\eR^2 \setminus \gamma(H')$ consists of two regions, a
bounded and an unbounded one. We say that a point $x$ lies inside $\gamma(H')$ 
if it lies in the bounded component of $\eR^2 \setminus \gamma(H')$. 
 
Recall that $Z$ denotes the subgraph of $H$ induced by the vertices
$\{a^{(1)},a^{(2)}\} \cup V(F) \cup V(C^{(1)}) \cup V(C^{(2)}) 
\cup \bigcup_{m=1}^4 V(P_m)$.
By construction, $Z$ is triangle free and has minimum degree at least two.
Since it is an induced subgraph of $H$, by Lemma~\ref{lem:trianglefreeplanar}, 
the points $p(v) : v \in V(Z)$ define a F\'ary embedding of $Z$.

Let us observe that in any planar embedding of $Z$ either 
$p(a^{(1)})$ lies inside $\gamma(F)$ 
%%%(meaning that it is contained in the bounded component of $\eR^2 \setminus \bicup_{e \in F} c(e)$) 
or $p(a^{(2)})$ lies inside $\gamma(F)$ -- otherwise we could not 
embed the paths $P_1,\dots, P_4$ without crossings (see figure~\ref{fig:addingpathsP}, right).
Without loss of generality it is $p(a^{(1)})$ in our F\'ary embedding.
We must then also have (see again figure~\ref{fig:addingpathsP}, right)
that $p(a^{(1)})$ lies inside $\gamma(C^{(1)})$.
%%%In other words, $p(a^{(1)})$ is contained in the bounded
%%%region of $\eR^2 \setminus \bigcup_{i=0}^N [p(c_i^{(1)}), p(c_{i+1}^{(1)})]$
%%%(where we set $c_{N+1}^{(1)} = c_0^{(1)}$ for convenience).
This also gives that $D(a^{(1)})$ is contained in the 
bounded region of $\eR^2 \setminus \bigcup_{v\in C^{(1)}} D(v)$. 
(In other words $D(a^{(1)})$ is ``surrounded" by the $D(c_i^{(1)})$s.)

%%For $0 \leq i \leq N$ let 
%%$H_i$ denote the subgraph of $H$ induced by
Recall that $H^{(i,1)}$ denotes the subgraph of $H$ induced by 
the vertices $V(F) \cup V(X_i^{(1)}) \cup \bigcup_{m=1}^4 V(Q_m^{(i,1)})$.
By construction $H^{(i,1)}$ is triangle free and of minimum degree at least two.
Hence Lemma~\ref{lem:trianglefreeplanar} again gives that
$p(v) : v \in V(H^{(i,1)})$ defines a F\'ary embedding of $H$. 
We already know that $p(a^{(1)})$ lies inside $\gamma(F)$. 
It now follows that $p(a^{(1)}), p(s_i^{(1)}), p(t_i^{(1)}), p(b_i^{(1)})$ must
lie on the outer face in the Far\'y embedding of $X_i^{(1)}$, because
otherwise we could not embed $Q_1^{(i,1)}, \dots, Q_4^{(i,1)}$ without
crossings (see figure~\ref{fig:addingpathsQ}, right).
Thus, by Lemma~\ref{lem:radiuscsmall} and the choice of $X=O_k$, we have 
that $r(c_i^{(1)}) < r( a^{(1)} ) / 1000$ for all $0 \leq i \leq N$.

Now set $p := p(a^{(1)})$, let 
$\ell_x$ be the horizontal line through $p$ and $\ell_y$ the vertical line through
$p$.
Observe that, because $D(a)$ is surrounded by the 
$D(c_i^{(1)})$s and $r(c_i^{(1)}) < r(a) /1000$ for all $i$, each of the four quadrants %%%regions
of $\eR^2 \setminus (\ell_x \cup \ell_y)$
contains one of the disks $D(c_i^{(1)})$.
(See figure~\ref{fig:lxly}.)

\begin{figure}[h!]
\begin{center}
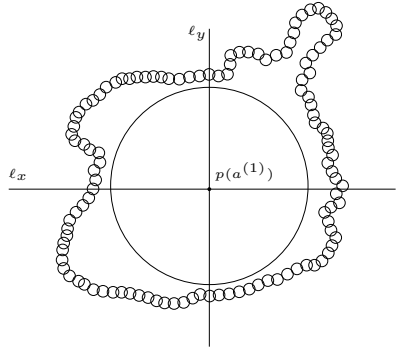
\end{center}
\caption{The $D(c_i^{(1)})$s surround $D(a^{(1)})$, so that each of the four quadrants 
defined by $\ell_x,\ell_y$ contains one of the $D(c_i^{(1)})$s.\label{fig:lxly}}
\end{figure}

Hence, any $W$ that intersects $D(v)$ for all $v \in V(H)$, 
has a point in each of the four regions of $\eR^2 \setminus (\ell_x \cup \ell_y)$; and
hence if such a $W$ is convex then it must contain $p$.
This proves part~\ref{itm:Hlemma.i} of the Lemma.

That part~\ref{itm:Hlemma.ii} holds is immediate from the choice of 
$p = p(a^{(1)})$ as the center of  $D(a^{(1)})$.
\end{proof}

\begin{lemma}\label{lem:DGembed}
Let $\Lcal = (\ell_1,\dots,\ell_n)$ be an oriented line arrangement and $\Scal \subseteq \Dcal(\Lcal)\cap\{-,+\}^n $.
There exists a disk graph $G$ on 
vertex set 
$V(G) = \{ v_i^-,v_i^+ : i=1,\dots,n\} \cup  \{ v^{(j)} : v \in V(H), j=1,\dots,|\Scal|\}$, such that %%$m := 2n+|\Scal|\cdot|v(H)|$ vertices, such that 
in any $\DG$-realization $(B(p(v),r(v)) : v\in V(G))$ of it, the 
oriented line arrangement $\Lcaltil = (\elltil_1,\dots, \elltil_n)$ defined by 

\[ 
\elltil_i^- := \{ z : w_i^T z < c_i \}, \]

\noindent 
where

\[ %%\begin{equation}\label{eq:wcdef} 
\begin{array}{l}
w_i := p(v_i^+)-p(v_i^-), \quad %\\
c_i := w_i^T {\Big(}\left(\frac{r(v_i^+)}{r(v_i^+)+r(v_i^-)}\right)p(v_i^-) 
+ {\Big(}\frac{r(v_i^-)}{r(v_i^+)+r(v_i^-)}{\Big)}p(v_i^+){\Big)}, 
\end{array} 
\] %%%\end{equation}

% \[ 
% \elltil_i^- := \{ z : (p(v_i^+)-p(v_i^-))^T z <  (p(v_i^+)-p(v_i^-))^T (\frac{r(v_i^+)}{r(v_i^+)+r(v_i^-)} p(v_i^+) 
% +
% \frac{r(v_i^-)}{r(v_i^+)+r(v_i^-)} p(v_i^-) ) \}, %\\
% \]

\noindent
has $\Scal \subseteq \Dcal(\Lcaltil)$. 
\end{lemma}

\begin{proof}
%%%Let $\Lcal$ be a realization of $\Dcal$, and l
Let $\Pcal = \{ p_1,\dots,p_{|\Scal|}\}$ be a set of points
such that $\Scal = \{ \sigma(p;\Lcal) : p \in \Pcal \}$.

Let $D_1^-,D_1^+,\dots,D_n^-,D_n^+$ be as provided by Lemma~\ref{lem:embeddisksR}, and let us set

\begin{equation}\label{eq:GDGdef}
\begin{array}{l}
D(v_{i}^s) := D_i^s  \quad \text{ for } s \in \{-,+\}, j=1,\dots,n, \\
\end{array} 
\end{equation}

\noindent
Now consider a $1 \leq j \leq |\Scal|$. Then $O_j := \bigcap \{ D_i^s : p_j \in D_i^s \}$ is open.
Hence we can place a suitably shrunken copy of a realization of $H$ inside $O_j$.
Let $H^{(j)}$ denote the copy of $H$ placed in $O_{j}$ and let 
$u^{(j)} : u \in V(H)$ denote the vertices of $H^{(j)}$.
%%%Observe that $D(u_1^{(j)})$ intersects $D(v_i^{s})$ if and only if
%%%$p_j \in \ell_i^{s}$.

Let $G$ be the corresponding intersection graph of disks, and let 
$(\Dtil(v) : v \in V(G))$ be an arbitrary realization of $G$ as a disk graph.
Let us write $\Dtil(v) = B(\ptil(v),\rtil(v))$ for all $v \in V(G)$.
For each $1\leq j \leq |\Scal|$, let 
$\ptil_j = p( ( \Dtil(u^{(j)}) : u \in V(H) ) )$ be the point 
provided by Lemma~\ref{lem:Hlemma}, applied to $H^{(j)}$.

Suppose that, for some $1 \leq j \leq |\Scal|$ and
$1\leq i \leq n$ and $s \in \{-,+\}$ we have
$p_j \in D(v_i^s)$.
Then we have that $D(u^{(j)}) \cap D(v_i^s ) \neq \emptyset$ for all $u \in V(H)$
by construction.
We must then also have $\Dtil(u^{(j)}) \cap \Dtil(v_i^s ) \neq \emptyset$ for all $u \in V(H)$,
because the $\Dtil$s and the $D$s define the same intersection graph.
Since $\Dtil(v_i^s)$ is convex, by the property of $\ptil_j$ certified by Lemma~\ref{lem:Hlemma}, 
we also have $\ptil_j \in \Dtil(v_i^s)$.

Now suppose that $p_j\not\in D(v_i^s)$. Then 
$p_j \in D(v_i^{-s})$ by choice of the $D(v_i^s)$s 
(Lemma~\ref{lem:embeddisksR}).
And thus, by the argument we just gave, $\ptil_j \in \Dtil(v_i^{-s})$.
Since $\Dtil(v_i^s)\cap\Dtil(v_i^{-s}) = \emptyset$ we 
thus have $\ptil_j\not\in\Dtil(v_i^s)$.
We have just proved that 

\begin{equation}\label{eq:pptiliff}
p_j \in D(v_i^s) \text{ if and only if }
\ptil_j \in \Dtil(v_i^s).
\end{equation}

\noindent
Let the oriented line arrangement $\Lcaltil = (\elltil_1,\dots,\elltil_n)$  be
as in the statement of the lemma.
Pick an arbitrary $z \in \Dtil(v_i^-)$.
Then we can write $z = \ptil(v_i^-) + \rtil(v_i^-) u$ with $\norm{u} < 1$.
Hence, with $w_i$ and $c_i$ as in the statement of the lemma, we have %writing $w_i := p(v_i^+)-p(v_i^-)$, we have

\[ \begin{array}{rcl}
w_i^T z
& = & 
w_i^T{\Big(}\ptil(v_i^-) + \rtil(v_i^-) u{\Big)} \\
& < & 
w_i^T {\Big(}\ptil(v_i^-) 
+ \rtil(v_i^-) \frac{w_i}{\norm{w_i}}{\Big)} \\
& \leq & 
w_i^T \left(\ptil(v_i^-) + 
\left(\frac{\rtil(v_i^-)}{\rtil(v_i^-)+\rtil(v_i^+)}\right)(\ptil(v_i^+) - \ptil(v_i^-)) \right) \\
& = & 
w_i^T {\Big(}\left(\frac{\rtil(v_i^+)}{\rtil(v_i^+)+\rtil(v_i^-)}\right)\ptil(v_i^-) 
+ {\Big(}\frac{\rtil(v_i^-)}{\rtil(v_i^+)+\rtil(v_i^-)}{\Big)}\ptil(v_i^+){\Big)} \\
& = & c_i, %%.
%%w_i^T{\Big(}p(v_i^-) + \lambda_i(p(v_i^+) - p(v_i^-)){\Big)} \\
%%& = & 
%%{\Big(}p(v_i^+) - p(v_i^-){\Big)}^T{\Big(}(1-\lambda_i)p(v_i^-) + \lambda_i p(v_i^+){\Big)},
\end{array} \]

\noindent
where we have used that $\norm{w} = \norm{\ptil(v_i^+) - \ptil(v_i^-)} \geq \rtil(v_i^+)+\rtil(v_i^-)$ 
(since $\Dtil(v_i^+)$ and $\Dtil(v_i^-)$ are disjoint)
in the third line.
We have just proved that $\Dtil(v_i^-) \subseteq \elltil_i^-$ for all $i$.
Completely analogously $\Dtil(v_i^+) \subseteq \elltil_i^+$ for all $i$.
By~\eqref{eq:pptiliff} and the choice 
of $D(v_i^s) := D_i^s$ 
(recall the $D_i^s$s are chosen such that $D_i^s \subseteq \ell_i^s$, and that 
either $p_j \in D_i^-$ or $p_j\in D_i^+$ for all $j$) we see that

\[ \sigma(p_j;\Lcal) = \sigma(\ptil_j;\Lcaltil), \]

\noindent
for all $j$. This proves the lemma.
\end{proof}

\noindent
Now we are finally in a position to prove the lower bound of Theorem~\ref{thm:DG}.

\begin{lemma}
$\fDG(n) = 2^{2^{\Omega(n)}}$.
\end{lemma}
 
\begin{proof}
It again suffices to prove that for every $k\in \eN$ there exists a
disk graph $G$ on $O(k)$ vertices with $\fDG(G) = 2^{2^{\Omega(k)}}$.
Let us thus pick an arbitrary $k \in \eN$, let $\Lcal, \Scal$ be as provided by Theorem~\ref{thm:sizeLline},
and let $G$ be as provided by Lemma~\ref{lem:DGembed}.
Then $|V(G)| = 2|\Lcal| + |V(H)|\cdot|\Scal| = O(k)$.
Let $(B(p(v),r(v)) : v \in V(G) )$ be an arbitrary 
$\DG$-realization of $G$
with $p(v) \in \{-m,\dots,m\}^2$ and $r(v) \in \{1,\dots,m\}$ for all $v$ for some integer $m\in\eN$.
Let $\Lcaltil$ be as defined in the statement of Lemma~\ref{lem:DGembed}.
Then we can write, with $w_i, c_i$ as in Lemma~\ref{lem:DGembed}: %%%\eqref{eq:wcdef}:

\[ \begin{array}{rcl}
\elltil_i
& = & 
\{ z : w_i^T z = c_i \} \\
%%\{ z :  (z_{2i}-z_{2i-1})^T z  (z_{2i}-z_{2i-1})^T (\lambda_iz_{2i}+(1-\lambda_i)z_{2i-1}) \} \\
%%& = & 
& = & 
\{ z :  (r(v_i^+)+r(v_i^-))w_i^T z = w_i^T (r(v_i^-)p(v_i^+)+r(v_i^+)p(v_i^-)) \} \\
& =: & 
\{ z : (w_i')^T z = c_i' \}. 
\end{array} \]

\noindent
Observe that the $w_i'$s have integer coordinates and the $c_i'$s are integers, whose absolute
values are all upper bounded by $8 m^3$. 
We can thus apply Lemma~\ref{lem:gridspan} to get that $2^{9/2} (8m^3)^6 
\geq \spanH(\Lcal) \geq 2^{2^k}$, and hence 
$m = 2^{2^{\Omega(k)}}$.
\end{proof}

\section{The lower bound for segment graphs\label{sec:SEGlb}}

An important tool in this section will be part (b) of the 
``order forcing lemma'' of Kratochv\'{\i}l and Matou{\v{s}}ek~\cite{KratochvilMatousek94}:

\begin{lemma}[\cite{KratochvilMatousek94}]\label{lem:orderforcing}
Let $G$ be a segment graph and $(S(v):v\in V(G))$ a
$\SEG$-realization such that all parallel segments are disjoint and no three
segments share a point.
Then there exists a segment graph $G'$ with $G \subseteq G'$ such that 
for every $\SEG$-realization $(\Stil(v) : v \in V(G'))$ there exists
an open set $O \subseteq \eR^2$ and a homeomorphism $\varphi: \eR^2 \to O$ such that
$\varphi[ S(v) ] \subseteq \Stil(v)$ for all $v \in V(G)$.
\end{lemma}

\begin{lemma}\label{lem:SEGembed}
Let $\Lcal = (\ell_1,\dots,\ell_n)$ be an oriented line arrangement
and $\Scal \subseteq \Dcal(\Lcal)$.
There exists a segment graph $G$ on
$\tel{\Lcal}+2\tel{\Scal}$ vertices such that for every
$\SEG$-realization $( S(v) : v \in V(G) )$ there is an oriented line arrangement
$\Lcaltil = (\elltil_1,\dots,\elltil_n)$ such that each line
$\elltil_i$ contains some segment $S(v)$ 
and $\Scal \subseteq \Dcal(\Lcaltil)$.
\end{lemma}

\begin{proof}
We start with the segment graph $G_0$ on the vertices
$x_1,\dots, x_4$, $y_1$,$y_2$,$t$,$m$,$b$ and the realization $S$ of it given in
figure~\ref{fig:SEGbase}.
%%Let $O$ denote the inside of the quadrilateral bounded by $S(y_1),S(x_1),S(y_2),S(x_4)$;
Let $O_t$ denote the inside of the quadrilateral bounded by $S(y_1),S(x_1),S(y_2),S(x_2)$;
let $O_m$ denote the inside of the quadrilateral bounded by $S(y_1),S(x_2),S(y_2),S(x_3)$;
and let $O_b$ denote the inside of the quadrilateral bounded by $S(y_1),S(x_3),S(y_2),S(x_4)$.
Let us list three key properties of the embedding $S$ for convenient future reference.

\begin{enumerate}
\item[\SSS{1}] $S(t) \subseteq O_t, S(m) \subseteq O_m, S(b) \subseteq O_b$;
\item[\SSS{2}] $S(y_1)$ and $S(y_2)$ each intersect 
$S(x_1),\dots, S(x_4)$ in the order of the indices;
\item[\SSS{3}] If $\ell$ is a line that intersects $S(y_1)$ between its 
intersection points with $S(x_2)$ and $S(x_3)$, and $\ell$ intersects
$S(y_2)$ between its intersection points with $S(x_2)$ and $S(x_3)$, then
$\ell$ separates $S(t)$ from $S(b)$.
\end{enumerate}

\begin{figure}[h!]
\begin{center}
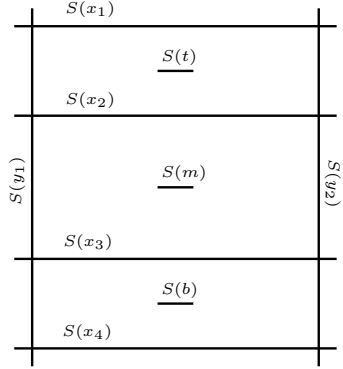
\end{center}
\caption{The $\SEG$-embedding of $G_0$ we are starting from.\label{fig:SEGbase}}
\end{figure}

% \noindent
% Let $O$ denote the inside of the quadrilateral bounded by
% $S(y_1), S(x_1),S(y_2),S(y_6)$; let 
% $O_b$ denote the inside of the quadrilateral bounded
% by $S(y_1),S(x_5),S(y_2),S(x_6)$; and let 
% $O_t$ denote the quadrilateral bounded by
% $S(y_1),S(x_1),S(y_2),S(x_2)$.
% Let us list some observations for convenient future reference:
% \begin{itemize}
% \item $S(t) \subseteq O_t, S(b) \subseteq O(b)$;
% \item $S(m) \subseteq O$;
% \item $S(y_1)$ hits $S(x_1),\dots, S(x_6)$ in this order, and none of
% $S(m), S(b), S(t)$;
% \item $S(y_3)$ hits $S(x_1),\dots, S(x_6)$ in this order, and none of
% $S(m), S(b), S(t)$;
% \item $S(m)$ hits $S(x_3)$ and $S(x_4)$ and none of
% $S(x_1),S(x_2),S(x_5),S(x_6),S(y_1),S(y_2)$.
% \end{itemize}

%%In this realization of $G_0$ the segments 

Let $G_1$ be the graph that Lemma~\ref{lem:orderforcing} constructs out of $G_0$, and 
consider an arbitrary $\SEG$-realization $S'$ of $G_1$.
Let us define $O_t',O_m',O_b'$ in the obvious way.
By construction of $G_1$, the properties \SSS{1}-\SSS{3}
hold also for $S'$ and $O_t',O_m', O_b'$.
(It can be seen that \SSS{1}-\SSS{3} are ``preserved under homeomorphism" in the
sense of Lemma~\ref{lem:orderforcing}).
Now let an oriented line arrangement $\Lcal = (\ell_1,\dots,\ell_n)$ and a set of sign vectors 
$\Scal \subseteq \Dcal(\Lcal)$ be given, and let $\Pcal = \{p_1,\dots,p_{|\Scal|}\}$ be such that 
$\sigma(\Pcal;\Lcal) = \Scal$.
By applying a suitable affine transformation if needed
(i.e.~we define a new oriented line arrangement and point set
by setting $\elltil_i^- := T[\ell_i^-]$ and $\ptil_j := T(p_j)$ -- 
clearly this does not affect sign vectors), we can assume without 
loss of generality that $\Pcal \subseteq O_m'$; that each line $\ell_i$ intersects 
$S'(m)$ and it intersects $S'(y_1)$ between the
intersection points with $S'(x_2)$ and $S'(x_3)$, and it
intersects $S'(y_2)$ between the intersection points with $S'(x_2)$ and $S'(x_3)$.
(So in particular $\ell_i$ separates $S'(t)$ from $S'(b)$ for all $1\leq i \leq n$.)
%%For convenience we will assume that 
%%$S(t) \subseteq \ell_i^+$ for all $i$.
%%%(Otherwise we would just have to add a minor amount of administration to our proof.)
For each $1 \leq i \leq n$ let $S'(v_i)\subseteq\ell_i$ denote the segment
between the intersection point of $\ell_i$ with $S'(y_1)$ and 
the intersection point of $\ell_i$ with $S'(y_2)$.

For each $1\leq j \leq |\Scal|$ let 
$S'(t_j)$ be a line segment between $p_j$ and a point on $S'(t)$; and let
$S'(b_j)$ be a line segment between $p_j$ and a point on $S'(b)$.
We let $G$ be the intersection graph of the segments $(S'(v) : v \in V(G_1) \cup \{v_1,\dots,v_n\} \cup 
\{b_j, t_j : 1 \leq j \leq |\Scal|\} )$ we just constructed.
(See figure~\ref{fig:SEGencode} for a depiction of this construction.)

\begin{figure}[h!]
\begin{center}
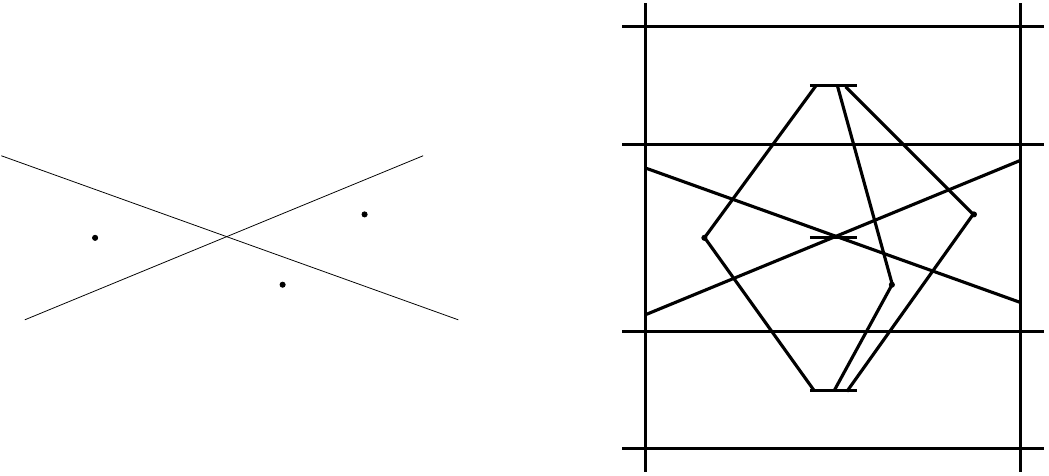
\end{center}
\caption{$\Lcal$ and $\Pcal$ (left) and the a segment graph $G$
we construct from them (right).\label{fig:SEGencode}}
\end{figure}

%\noindent
Now let $(\Stil(v) : v \in V(G))$ be an arbitrary realization of $G$ as
a segment graph, and let $\Otil_t, \Otil_m, \Otil_b$ be defined
in the obvious way. Again \SSS{1}-\SSS{3} hold.
%%%Again the properties listed above for $S(v) : v \in V(G_0)$ are also true in $\Stil$.
Let $\elltil_i$ denote the line that contains $\Stil(v_i)$ for $i=1,\dots,n$.
Since $\Stil(v_i)$ intersects all of $\Stil(m), \Stil(y_1),\Stil(y_2)$ and none of 
$\Stil(x_1),\dots,\Stil(x_4)$ (by construction of $G_1$), it must hold that 
$\Stil(v_i)$ intersects $\Stil(y_1)$ between 
the intersection points with $\Stil(x_3)$ and $\Stil(x_4)$; and $\Stil(v_i)$ intersects $\Stil(y_2)$ between 
the intersection points with $\Stil(x_3)$ and $\Stil(x_4)$.
Hence $\elltil_i$ separates $\Stil(t)$ from $\Stil(b)$.
Let us orient the lines $\elltil_i$ such that $\elltil_i^+ \supseteq \Stil(t)$ if and only if
$\ell_i^+ \supseteq S'(t)$.
For $1\leq j \leq |\Scal|$ let 
$\ptil_j$ be a point of $\Stil(b_j)\cap\Stil(t_j)$
(such a point exists as $S'(b_j),S'(t_j)$ intersect), and
let us write $\Pcaltil = \{\ptil_1,\dots,\ptil_{|\Scal|}\}$.
We claim that 

\begin{equation}\label{eq:SStilSEG}
\sigma(\Pcal;\Lcal) = \sigma(\Pcaltil;\Lcaltil).
\end{equation}

\noindent
To see that~\eqref{eq:SStilSEG} holds, pick an arbitrary $1\leq i \leq n$ and an 
arbitrary $1\leq j \leq |\Scal|$.
Suppose that $p_j$ and $S'(t)$ lie on the same side of $\ell_i$.
Since $\Stil(t_j)$ hits both $\ptil_j$ and $\Stil(t)$ but it does not
hit $\elltil_i$, we see that $\ptil_j$ and $\Stil(t)$ are also on the same side
of $\elltil_i$.
Similarly, if $p_j$ and $S'(b)$ are on the same side of $\ell_i$ then
$\ptil_j$ and $\Stil(b)$ are on the same side of $\elltil_i$.
%%Thus $\sigma(p_j;\Lcal)_i = \sigma(\ptil_j;\Lcal)_i$.

By choice of (the orientation of) $\Lcaltil$, this proves~\eqref{eq:SStilSEG}
and hence the lemma.
\end{proof}

\begin{lemma}
$\fSEG(n) = 2^{2^{\Omega(n)}}$.
\end{lemma}
 
\begin{proof}
It again suffices to prove that for every $k\in \eN$ there exists a
segment graph $G$ on $O(k)$ vertices with $\fSEG(G) = 2^{2^{\Omega(k)}}$.
Let us thus pick an arbitrary $k \in \eN$, let $\Lcal, \Scal$ be as provided by Theorem~\ref{thm:sizeLline},
and let $G$ be as provided by Lemma~\ref{lem:SEGembed}.
Then $|V(G)| = |V(G_1)| + |\Lcal| + 2|\Scal| = O(k)$.
Let $(S(v) : v \in V(G))$ be an arbitrary 
$\SEG$-realization of $G$ such that, for some $m\in\eN$, we 
can write $S(v) = [a(v),b(v)]$ with $a(v), b(v) \in \{-m,\dots,m\}^2$ 
for all $v \in V(G)$.
%%Consider a line $\ell_i$ of $\Lcal$, and suppose that $\ell_i$ contains $S(v)$.
Let us order $V(G)$ arbitrarily as $v_1,\dots,v_n$ and let 
the oriented line arrangement $\Lcal = (\ell_1,\dots,\ell_n)$
be such that $\ell_i$ contains $S(v_i)$ for $i=1,\dots,n$
(the orientation does not matter in the sequel).
%%If $\ell(v)$ is the line that contains $S(v)$ 
Then we can write

\[ \ell_i = \{ z : w^T z = c \}, \]

\noindent
with 

\[ 
%%%%w = ((b(v))_y-(a(v))_y, (b(v))_x-(a(v))_x)^T, \quad c = w^T a(v). 
w = \left(\begin{array}{c}(b(v_i))_y-(a(v_i))_y \\ (b(v_i))_x-(a(v_i))_x \end{array} \right), \quad c = w^T a(v_i). 
\]

\noindent
Thus the coordinates of $w$ and $c$ are integers whose absolute values are upper bounded by 
$4 m^2$. We can again apply Lemma~\ref{lem:gridspan} to see that 
$2^{7/2} (4m^2)^6 \geq \spanH(\Lcal) \geq 2^{2^k}$, and hence 
$m = 2^{2^{\Omega(k)}}$.
\end{proof}

\section{Proofs of the upper bounds\label{sec:ub}}

The {\em order type} of a point configuration $\Pcal = (p_1,\dots,p_n)$
stores for each triple of indices $1 \leq i_1<i_2<i_3 \leq n$ whether the 
points $p_{i_1}, p_{i_2},p_{i_3}$ are in clockwise position, in counter clockwise position or collinear.
Recall that a point configuration is in {\em general position} if no three points are collinear.
We need the following result of Goodman, Pollack and Sturmfels~\cite{GPS90}, stated here only for 
two dimensions.

\begin{theorem}[\cite{GPS90}]\label{thm:GPS}
Let $f(n)$ denote the least $k$ such that every order type of $n$ points
in general position in the plane can be realized by points on $\{1,\dots,k\}^2$.
Then $f(n) = 2^{2^{\Theta(n)}}$.
\end{theorem} 

\noindent
Observe that every segment graph has a realization in which  
the endpoints of the segments are in general position.
(To see this, start with an arbitrary realization.
By making the segments slightly longer if needed we can ensure that
every two intersecting segments intersect in a point that is interior to both. 
Now we can perturb the endpoints very slightly so that they are in
general position and the intersection graph of segments remains the same.)
Let us also observe that we can tell whether two segments $[a,b], [c,d]$ intersect or not
from the order type of $(a,b,c,d)$, unless $a,b,c,d$ are all collinear.
%%%,first note we can assume no intersection point 
%%%%in which case we can tell from the order type of $(a,b,c,d,e)$ if $e$ is not on the same
%%%line as $a,b,c,d$. 
Hence if $( S(v) : v \in V(G) )$ is a $\SEG$-realization of $G$ whose endpoints are in general
position, then any point configuration with the same order type as the endpoints of the
segments %%%(adding one extra point if they're collinear) 
also gives a $\SEG$-realization of $G$.
%%(unless all the segments are collinear -- but this case is easy to deal with).
Therefore Theorem~\ref{thm:GPS} immediately implies the upper bound
in Theorem~\ref{thm:SEG}.

\begin{corollary}
$\fSEG(n) = 2^{2^{O(n)}}$.
\end{corollary}

The upper bounds in Theorems~\ref{thm:DG},~\ref{thm:UDG}
are relatively straightforward consequences of a result of Grigor'ev and Vorobjov
that was also the main ingredient in the proof of the upper bound in Theorem~\ref{thm:GPS}. 
The following is a reformulation of Lemma 10 in~\cite{grigorevvorobjov88}:

\begin{lemma}[\cite{grigorevvorobjov88}]\label{prop:grigor}
For each $d \in \eN$ there exists a constant $C = C(d)$ such that
the following hold.
Suppose that $h_1, \dots, h_k$ are polynomials in $n$ variables
with integer coefficients, and degrees $\deg(h_i) < d$.
Suppose further that the bit sizes of the all coefficients are less than $B$.
If there exists a solution $(x_1,\dots,x_n) \in \eR^n$
of the system $\{h_1\geq 0, \dots, h_k \geq 0\}$, then 
there also exists one with 
$|x_1|, \dots, |x_n| \leq \exp[ (B+\ln k)C^n]$.
\end{lemma}

\begin{lemma}\label{lem:UDGub} $\fUDG(n) = 2^{2^{O(n)}}$.
\end{lemma}

\begin{proof}
Let $G$ be a unit disk graph on $n$ vertices.
Any $\UDG$-realization is a nothing more than a solution 
$(x_1,y_1,\dots,x_n,y_n, r) \in \eR^{2n+1}$ of the system
of polynomial of inequalities

\begin{equation}\label{eq:UDGsyst} 
\begin{array}{ll}
(x_i-x_j)^2 + (y_i-y_j)^2 < r^2, & \text{ for all } ij \in E(G), \\
(x_i-x_j)^2 + (y_i-y_j)^2 \geq r^2, & \text{ for all } ij \not\in E(G), \\
r > 0.
\end{array} 
\end{equation}

\noindent
Observe that any solution of~\eqref{eq:UDGsyst} can be perturbed
to a solution in which all inequalities are strict 
(if we fix $r'=r$ and set $x_i' = \lambda x_i, y_i'=\lambda y_i$ for all
$1\leq i \leq n$ then, if we chose $\lambda > 1$ but very very close to 1, 
then we have a new solution in which all inequalities are strict).
Similarly, if we now multiply all variables by the same (very large) scalar $\mu$
we get a solution of:

\begin{equation}\label{eq:UDGsyst2} 
\begin{array}{ll}
(x_i-x_j)^2 + (y_i-y_j)^2 \leq (r-10)^2, & \text{ for all } ij \in E(G), \\
(x_i-x_j)^2 + (y_i-y_j)^2 \geq (r+10)^2, & \text{ for all } ij \not\in E(G), \\
r \geq 100.
\end{array} 
\end{equation}

\noindent
This is a system of $1+{n\choose 2}$ polynomial inequalities of degree less than 3 in 
$2n+1$ variables, with all coefficients small integers.
Since the system has a solution, by lemma~\ref{prop:grigor}, there exists a solution to this system with 
all numbers less than $\lfloor\exp[ \gamma^n ]\rfloor$ in absolute value for some $\gamma$
(we absorb the factor $\ln( 1+{n\choose 2} ) + O(1)$ by taking $\gamma > C$).
Let us now round down all numbers to the next integer, i.e.~we set
$\xtil_i := \lfloor x_i\rfloor, \ytil_i:=\lfloor y_i\rfloor, 
\rtil := \lfloor r\rfloor$.
If $ij \in E(G)$ then we have

\[ \begin{array}{rcl}
(\xtil_i-\xtil_j)^2 + (\ytil_i-\ytil_j)^2
& \leq & 
(|x_i-x_j| + 1)^2 + (|y_i-y_j| + 1)^2 \\
& = & 
(x_i-x_j)^2 + (y_i-y_j)^2
+ 2 (|x_i-x_j| + |y_i-y_j|) + 2 \\
& \leq & 
(r-10)^2 + 4 (r-10) + 2 \\
& = & 
r^2 - 16 r + 62 \\
& \leq & 
(r-1)^2 \\
& < & 
\rtil^2
\end{array}\]

\noindent
(Here we have used $|x_i-x_j| + |y_i-y_j| \leq 2\sqrt{(x_i-x_j)^2 + (y_i-y_j)^2} \leq 2(r-10)$ in the 
third line; and in the fifth line we used that $r\geq 100$.)
Similarly, if $ij\not\in E(G)$ then 

\[ \begin{array}{rcl}
(\xtil_i-\xtil_j)^2 + (\ytil_i-\ytil_j)^2
& \geq & 
(|x_i-x_j| - 1)^2 + (|y_i-y_j| - 1)^2 \\
& = & 
(x_i-x_j)^2 + (y_i-y_j)^2
- 2 (|x_i-x_j| + |y_i-y_j|) + 2 \\
& \geq & 
(x_i-x_j)^2 + (y_i-y_j)^2
- 4 \sqrt{(x_i-x_j)^2 + (y_i-y_j)^2} + 2 \\
& \geq & 
(r+10)^2 - 4(r+10) + 2 \\
& = & 
r^2 + 16 r + 62 \\
& \geq & 
r^2 \\
& \geq & 
\rtil.
\end{array} \]

\noindent
(Here we have used that 
$d^2 - 4d + 2$ is increasing for $d\geq 2$ in the fourth line.) 
Thus $(\xtil_1,\ytil_1,\dots,\xtil_n,\ytil_n,\rtil)$ is a 
solution of~\eqref{eq:UDGsyst}.
Since all variables of this solution are integers
of absolute value at most 

\[ 
\exp[\gamma^n] = 2^{\log(e) \cdot \gamma^n} = 2^{2^{n \log\gamma + \log\log e}} = 2^{2^{O(n)}}, \]

\noindent
this proves the lemma.
\end{proof}

\noindent
The proof of the upper bound in Theorem~\ref{thm:DG} is very similar to the
proof of Lemma~\ref{lem:UDGub} and we therefore omit it.

\begin{lemma}
$\fDG(n) = 2^{2^{O(n)}}$. \noproof
\end{lemma}

%%%%%%%%%%%%%%%%%%%%%%%%%%%%%%%%%%%%%%%%%%%%%%%%%%%%%%%%%%%%%%%%%%%%
%%%%%%%%%%%%%%%%%%%%%%%%%%%%%%%%%%%%%%%%%%%%%%%%%%%%%%%%%%%%%%%%%%%%

\section*{Acknowledgements}

We thank Erik Jan van Leeuwen and Tony Huyhn for helpful discussions. 
We thank Jan Kratochv{\'{\i}}l and Ji{\v{r}}{\'{\i}} Matou{\v{s}}ek for making their 1988 
preprint available to us.
We also thank Peter Shor for making a copy of his 1991 paper available to us.

\bibliographystyle{plain}
\bibliography{ReferencesDGs}

\end{document}